\documentclass[onefignum,onetabnum]{siamart190516}

\usepackage{lipsum}
\usepackage{amsfonts}
\usepackage{graphicx}
\usepackage{epstopdf}
\usepackage{algorithmic}

\usepackage{animate}
\usepackage{amsmath}
\usepackage{amssymb}
\usepackage{graphics}
\usepackage{graphicx}
\usepackage{footnote}
\usepackage{textcomp}
\usepackage{mathrsfs}
\usepackage{epstopdf}
\usepackage{array}
\usepackage[maxfloats=99]{morefloats}
\usepackage{url}
\usepackage{cases}
\usepackage{mathscinet}
\usepackage[normalem]{ulem}
\usepackage{algorithmic, algorithm}
\usepackage{color}
\usepackage{cite}

\ifpdf
  \DeclareGraphicsExtensions{.eps,.pdf,.png,.jpg}
\else
  \DeclareGraphicsExtensions{.eps}
\fi

\newsiamremark{remark}{Remark}
\newsiamremark{hypothesis}{Hypothesis}
\crefname{hypothesis}{Hypothesis}{Hypotheses}
\newsiamthm{claim}{Claim}

\headers{Optimal maximum norm estimates for virtual element methods}{Wen-ming He and Hailong  Guo}

\title{Optimal maximum norm estimates for virtual element methods}

\author{Wen-ming He\thanks{Department of Mathematics, Lingnan Normal University,
Zhanjiang, Guangdong 524000,  P. R. China (\email{he\_wenming@aliyun.com}) }
\and 
Hailong Guo\thanks{Corresponding author. School of Mathematics and Statistics,  The University of Melbourne,  Parkville, VIC 3010, Australia   
  (\email{hailong.guo@unimelb.edu.au}),}
}

\usepackage{amsopn}

\ifpdf
\hypersetup{
  pdftitle={Optimal maximum norm estimates},
  pdfauthor={Wen-ming He and Hailong Guo}
}
\fi

\usepackage{subcaption}

\usepackage{multirow}

\graphicspath{{fig/}}

\begin{document}

\maketitle

\begin{abstract}
The maximum norm error estimations for virtual element methods are studied. To establish the error estimations, we prove higher local regularity based on delicate analysis of Green's functions and high-order local error estimations for the partition of the virtual element solutions.  The maximum norm of the exact gradient and the gradient of the projection of the virtual element solutions are proved to achieve optimal convergence results. For high-order virtual element methods, we establish the optimal convergence results in $L^{\infty}$ norm. Our theoretical discoveries are validated by a numerical example on general polygonal meshes. 
 
\end{abstract}

\begin{keywords}
  virtual element method, maximal error estimate, Green's function
  \end{keywords}

\begin{AMS}
 65N30, 65N25,65N15.  
\end{AMS}

\section{Introduction}
\label{sec:int}

%
%
%
%
%

In recent years, there has been a surge of interest in developing numerical methods for numerically solving partial differential equation using general polygonal/polyhedral meshes.  
The construct of finite element shape functions on convex polygons was first articulated by Wachspress (1971) \cite{Wa1971} and popularized in his book: A Rational Finite Element Basis \cite{Wa1975}.
Since that,  considerable literature has grown up around the theme of developing finite element/difference methods using general polygons/polyhedra.     Famous  examples include
  the  polygonal finite element methods\cite{ST2004, SM2006},   mimetic finite difference methods\cite{BLM2014,Sh1996,KR2003, HS1999, SS1996},  hybrid high-order methods\cite{DEL2014, DE2015, Lema2020},  polygonal discontinuous Galerkin methods \cite{MWWY2014}, etc.  The interesting readers are referred to  \cite{MRS2014} for the recent review.

  Virtual element methods were originated from developing higher-order mimetic finite difference methods\cite{BBL2009,BLM2011}  using the framework of finite element methods.  The key idea is to use non-polynomial basis functions which are similar to the polygonal finite element methods\cite{ST2004} or the extended/generalized finite element methods\cite{FrBe2010, BaOs1983}. Different from other numerical methods using general polygons/polyhedrons,  virtual element methods preserve the form of classical finite element methods on simplexes while use general polygons/polyhedrons. The beauty of virtual element methods is that the non-polynomial basis functions are never explicitly constructed (or needed). The local stiffness (or mass) matrix is assembled only using the polynomial basis functions and suitable projections \cite{AHMAD2013}. This capability allows virtual element methods to handle more general continuity requirements like $C^r$ continuity with $r\ge 1$. The usage of polygonal/polyhedral meshes makes virtual element methods handle hanging nodes naturally and simplifies the procedure of adaptive mesh refinement.

  Virtual element method was firstly proposed by Beir\~{a}o da Veiga etc. In their seminal study \cite{BBCMMR2013},  Beir\~{a}o da Veiga etc. used the Poisson equation to illustrate the abstract framework of constructing and analyzing arbitrary order virtual element methods using a local $H^1$-projector. Within the abstract framework, the virtual element method is proven to be convergent optimally in both the $H^1$ norm and $L^2$ norm.  In a study conducted by Ahmad etc.\cite{AHMAD2013}, it was shown that a local $L^2$-orthogonal projector can be easily computed using the local $H^1$-projector by slightly changing the local virtual element space (or non-polynomial basis functions), which may facilitate the treatment of lower order terms and $L^2$ error analysis.  In \cite{BBMR2016}, Beir\~{a}o Da Veiga etc. investigated the virtual element methods for general second-order elliptic equations with variable coefficients. 
  Thereafter,  there has been a lot of study of virtual element for other equations, just to name a few, see\cite{BBMR2016,CMS2017,ADLP2017,BM2013,  CGM2016, BBDMR2018, ZZMC2019
  }.  In the most recent work \cite{GuXZ2019}, Guo etc. examined the superconvergence property of the linear virtual element method and its corresponding post-processing technique.

The greater part of the literature on virtual element methods seems to have been based on the optimal error estimate in the energy norm. 
Up to now, there are only limited studies focusing on the analysis of the maximal norm error for virtual element methods in the literature.  The first attempts on maximal norm error estimation for virtual element methods was presented by Brenner and Sung in \cite{BS2018}, where a suboptimal $L^{\infty}$ error estimate was obtained.   To our best knowledge,   there is no optimal maximal norm error estimation for virtual element methods. In contrast, the maximal norm error estimates \cite{ScWa1978, HeZZ2018} for classical numerical methods like finite element methods\cite{BS2008, Ci2002} and finite volume methods\cite{ZhZo2015}  have been  in the maturity stage.

The main purpose of this paper is to establish the optimal maximum norm estimates for virtual element methods under the general setting.  For the maximum norm estimates for the classical numerical methods, the main ingredient is the inverse error estimate \cite{BS2008, Ci2002}. However, for the virtual element solutions, they are no longer piecewise polynomials even though we don't explicitly  construct their non-polynomial parts.  To overcome the difficulty for maximum error for gradients, we consider the gradient of projection of the virtual element solutions that are actually  what we can compute and use the inverse estimate for polynomials on general polygonal domain\cite{DPDR2020}.  For the maximum norm error for functions values, we bypass those difficulties by establishing an inverse error estimate using only the maximum principle for $H^1$ functions \cite{BGS2017}.  Our error estimations use the regularity results based on the delicate estimation of Green's functions.  By using a special partition of unity, we can prove the high-order local error estimates.  It is worth pointing out that the high-order local $L^2$ error works except linear virtual element method because of the $2k$ conjecture  \cite{ChHu2013}.  The established high-order local error estimates pave the way for our optimal maximum norm error estimates.

The rest of the paper is organized as follows. The model problem and its related Sobolev spaces are introduced in Section \ref{sec:model}.  The construct of virtual element methods is given in Section \ref{sec:vem}. Then,  the presentation and proof  our main results  are given  in Section \ref{sec:main}.   Two numerical examples are provided to verify our theoretical results  in Section 6.  Finally, some conclusion is drawn in Section \ref{sec:con}.

\section{Model problem}
\label{sec:model}
In this paper, the standard notation for Sobolev spaces as in \cite{BS2008, Ci2002, Evans2008} will be adopted. 
Let $\Omega\subset \mathbb{R}^2$ be a bounded polygonal domain with Lipschitz boundary $\partial\Omega$. 
For any subdomain  $\mathcal{D}$ of $\Omega$ and $1\le p \le \infty$,  let $W^{k,p}(\mathcal{D})$ denote the Sobolev space with norm
$\|\cdot\|_{W^{k, p}(\mathcal{D})} $ and seminorm $|\cdot|_{W^{k, p}(\mathcal{D})}$.
When $p=2$,  $W^{k,2}(\mathcal{D})$ is abbreviated as  $H^{k}(\mathcal{D})$ and $p$ is omitted from its norms.
Similarly,  $W^{0,p}(\mathcal{D})$ is abbreviated as $L^{p}(\mathcal{D})$.
 $(\cdot, \cdot)_{\mathcal{D}}$ denotes the standard $L^2$ inner product in $\mathcal{D}$ and the subscript
 is ignored when $\mathcal{D} = \Omega$.  Let $\mathbb{P}_k(\mathcal{D})$ denote  the space of polynomials of degree less than or equal to $k$ over $\mathcal{D}$.  In this paper, the letter $C$ or $c$ denotes a generic positive constant which may be different at different occurrences. 
 We will use $x \lesssim y$ to denote $x \le Cy$ for some constant $C$ independent of the mesh size.

We take the following second order elliptic equation as our model equation
\begin{equation}
\left\{
\begin{split}
  -\nabla\cdot (\alpha \nabla  u) =f, &\text{    in }\Omega,\\
    u=0, & \text{    on }\partial\Omega,
\end{split}	
\right.
\label{equ:model}
\end{equation}
where  the diffusion coefficient matrix $\alpha$ is a constant coefficient matrix  satisfying that   $ \alpha_0 \|\mathbf{x}\|^2 \le \mathbf{x}^T \alpha \mathbf{x} \le \alpha_1 \|\mathbf{x}\|^2  $ for some positive constant $\alpha_0$ and $\alpha_1$.  To simplify the notation, we denote the differential operator as $\L := -\nabla\cdot (\alpha \nabla). $

The variational formulation of \eqref{equ:model}  is to find  $u\in H^1_0(\Omega)$ such that
\begin{equation}\label{equ:varform}
a(u,  v) = (f, v), \quad \forall v\in H^1_0(\Omega),
\end{equation}
where  the bilinear form $a(\cdot, \cdot)$ is defined to be 
\begin{equation}\label{equ:bilinear}
a(w,v) = (\alpha\nabla w, \nabla v). 
\end{equation}
The boundedness of the coefficient matrix $\alpha$ and the Lax-Milgram  theorem imply the variational problem \eqref{equ:varform} has a unique solution $u$.

\section{Virtual element methods}
\label{sec:vem}
One of the important  merits of  virtual element methods is that they allow for general polygonal partitions of the computational domain $\Omega$. 
 Suppose   $\mathcal{T}_h$,   consisting of non-overlapping polygonal elements $E$, is a partition of $\Omega$.
For any element $E\in \mathcal{T}_h$, its diameter is denoted by
$h_E$.  The largest element diameter  of $\mathcal{T}_h$ is denoted by $h$.   Without  loss of generality,   we assume that there exists $\rho \in (0,1)$ such that
 the mesh $\mathcal{T}_h$
satisfies the following two assumptions \cite{BGS2017, BBMR2016}:
\begin{itemize}
\item [(i).]  every element $E$ is star-shaped with respect to every point of a disk $D$ of radius $\rho h_E$;
\item [(ii).] every edge $e$ of $E$ has length $|e| \ge \rho h_E$.
\end{itemize}

For a polygonal element $E$,  we denote the barycenter of $E$ by $x_{E}$.  For any nonnegative integer $r\in\mathbb{N}$,  let $\mathcal{M}_r(E)$ be the set of polynomials defined as
\begin{equation}\label{equ:mom}
\mathcal{M}_r(E):=\left\{m|m=\left(\frac{\mathbf{x}-\mathbf{x}_E}{h_E}\right)^{\mathbf{s}} \text{ for }  |\mathbf{s}|\le r\right\},
\end{equation}
with $\mathbf{s}=(s_1,s_2)\in\mathbb{N}^2$ being the multi-index and $|\mathbf{s}|=s_1+s_2$. We can check that $\mathcal{M}_r(E)$ is a basis for the space of polynomials of degree $\le r$ on $E$.  Furthermore, we define the subspace $\mathcal{M}^*_r(E)$  of $\mathcal{M}_r(E)$  as 
\begin{equation}\label{equ:mmom}
\mathcal{M}^*_r(E):=\left\{m|m=\left(\frac{\mathbf{x}-\mathbf{x}_E}{h_E}\right)^{\mathbf{s}} \text{ for }  |\mathbf{s}| = r \right\}.
\end{equation}
 To define the virtual element space, we begin with defining the
local virtual element spaces on each element. For any positive integer $k${\color{red},}  let
\begin{equation}
\mathbb B_{k}(\partial E):=\{v\in C^0(\partial E):v|_e\in\mathbb P_k(e),\quad \forall e\in \partial E\}.
\end{equation}
Then, the enlarged local virtual element space $V(E)$\cite{AHMAD2013} on the element $E$ can be defined as
\begin{equation}
V_k(E)=\{v\in H^1(E):v|_{\partial E}\in \mathbb B_{k}(\partial E),\quad\Delta v|_E \in \mathbb{P}_{k-2}(E) \},
\end{equation}
 with $\mathbb{P}_{-1}(E) = \{ 0\}$.

The soul of virtual element methods is that the non-polynomial basis functions are never explicitly constructed  and needed.  This is made possible
by introducing the  projection operator $\Pi^{\nabla}_k$.  For any function $v^{h} \in V(E)$,  its projection  $\Pi^{\nabla}_kv^{h}$ is defined to satisfy the following
orthogonality:
\begin{equation}\label{equ:orthproj}
(\nabla p, \nabla (\Pi_k^{\nabla} v^{h} - v^{h}))_{ E} = 0, \quad \forall p\in \mathbb{P}_k(E),
\end{equation}
plus (to take care of the constant part of $\Pi_k^{\nabla}$):
\begin{equation}
\int_{\partial E}(\Pi_k^{\nabla} v^{h}-v^{h})ds=0, \quad \text{ for } k = 1,
\end{equation}
or 
\begin{equation}
\int_{ E}(\Pi_k^{\nabla} v^{h}-v^{h})d\mathbf{x}=0, \quad \text{ for } k \ge 2.
\end{equation}

The modified local virtual element space  \cite{AHMAD2013} is defined as
\begin{equation}
W_k(E)=\{v^{h}\in V_k(E): (v^{h}- \Pi_k^{\nabla}v^{h}, q^*) = 0, \, \forall q^*\in  \mathcal{M}_{k-1}^*(E)\cup\mathcal{M}_k^*(E)\}.
\end{equation}
Then,  the  virtual element space  \cite{AHMAD2013, BBMR2014} is
\begin{equation}
V_{h}=\{v\in H^1(\Omega):v|_{ E}\in W_k(E), \,\forall E\in \mathcal{T}_h\}.
\end{equation}
Furthermore, let $V_{h,0} = V_{h}\cap H^1_0(\Omega)$
be the subspace of $V_{h}$ with homogeneous boundary condition.

Similarly, we can define the $L^2$ projection $\Pi_k^0$  as
\begin{equation}\label{equ:l2proj}
( p, \Pi_k^{0} v^{h} - v^{h})_{ E} = 0, \quad \forall p\in \mathbb{P}_k(E).
\end{equation}
For the linear virtual element method\cite{AHMAD2013, BBMR2014}, these two projections are equivalent, i.e. $\Pi_k^{\nabla} =\Pi_k^0$.

On each element $E\in \mathcal{T}_h$, we can define the following discrete bilinear form
\begin{align}
a^E_h(u^{h},v^{h})=(\alpha\nabla \Pi_k^{\nabla} u^{h}, \nabla \Pi_k^{\nabla} v^{h})_E+S^E(u^{h}-\Pi_k^{\nabla} u^{h}, v^{h}-\Pi_k^{\nabla} v^{h})
\end{align}
for any $ u^{h},v^{h}\in V(E)$.  The discrete bilinear form $S^E$ is symmetric,   positive definite, and continuous, which is also fully computable using only
the degrees of freedom of $u^{h}$.    The readers are referred to \cite{AHMAD2013,BBCMMR2013, BBMR2014} for the detail definition of $S^E$,  which is selected to make $a^E_h(\cdot, \cdot)$ satisfy 
\begin{itemize}
	\item $k$-Consistency: For  all $p\in\mathbb{P}_k(E)$ and all $v_h\in W_k(E)$,
	\begin{equation}
		a_h^E(p, v_h) = a^E(p, v_h),
	\end{equation}
	where $a^E(u,v) = \int_{E}\alpha\nabla u\cdot \nabla v d\mathbf{x}$. 
	\item Stability:  There exist two positive constants $\alpha_*$ and $\alpha^*$, independent of $h$ and $E$, such that 
	\begin{equation}
		\alpha_*a^E(v_h,v_h)\le a_h^E(v_h, v_h) \le \alpha^* a^E(v_h, v_h), \quad \forall v_h \in W_k(E).
	\end{equation}
\end{itemize}

Then, we can define the  discrete bilinear form $a_h(\cdot, \cdot)$:
\begin{align}
a_h(u^h,v^{h})=\sum_{E\in\mathcal{T}_h}a^E_h(u^h,v^{h}),
\end{align}
for any $ u^h,v^{h}\in V_{h}$.
The virtual element method for the model  problem \eqref{equ:model} is to  find $u^h \in V_{h, 0}$ such that
\begin{equation}\label{equ:vem}
a_h(u^h,v^{h})=(f,\Pi_{k-1}^0 v^{h}),\quad \forall v^{h} \in V_{h,0}.
\end{equation}

Assume $\chi_1, \cdots, \chi_{N_E}$ are the basis functions of the dual space of $W_k(E)$.  Define the local interpolation  $I_{h,E}w \in W_k(E)$  of a smooth enough function $w$ as a 
\begin{equation}
	\chi_i(w-I_{h,E}w) = 0.
\end{equation}
The global interpolation operator $I_h: C^0(\Omega)\rightarrow V_h$ is defined as
\begin{equation}\label{equ:interpolation}
	I_hw|_{E} := I_{h,E}w,\quad  \forall w\in C^0(\Omega). 
\end{equation}

For the virtual element method \eqref{equ:vem}, \cite{AHMAD2013, BBCMMR2013, BBMR2016} established the following convergence results in $L^2$ and $H^1$ norms.
\begin{theorem} \label{thm:h1error}
 Let $u$ be the solution to the problem \eqref{equ:model}, and
let  $u_{h}\in V_{h,0}$  be the solution of the discretized problem. Assume further that $\Omega$ is convex, and that the exact
solution $u(x)$  belongs to $H^{k+1}(\Omega)$. Then the  following estimate holds:
\begin{align}\label{equ:homoerror}
h\|u-u^h\|_{H^{1}(\Omega)} + \|u-u^h\|_{L^2(\Omega)}\lesssim h^{k+1}\|u\|_{H^{k+1}(\Omega)}.
\end{align}
\end{theorem}

In the proof of maximum norm error estimate, we  shall introduce  second-order elliptic equations with inhomogeneous boundary condition. For the inhomogeneous boundary value problem, we have the following error estimate.

\begin{theorem} \label{thm:inhomo}
 Assume $\phi\in H^1(\Omega)$ is the solution for the following second-order elliptic equation with inhomogeneous boundary condition:
 \begin{eqnarray}
\begin{cases}
\L\phi\equiv -\nabla\cdot(\alpha\nabla \phi(\mathbf{x}))=f(\mathbf{x}),\quad &\text{in}\quad \Omega,\\
\phi(\mathbf{x})=g(\mathbf{x}), &\text{on}\quad \partial\Omega.\label{equation1.1tt}
\end{cases}
\end{eqnarray}

We further assume that $\partial\Omega$ is smooth enough such that $\phi\in H^{k+1}(\Omega)$ and 
  $g\in  H^{k+\frac{1}{2}}(\partial\Omega)$. Let $\phi^h$ be the $k$th order virtual element approximation of $\phi$, there holds
 \begin{equation}\label{equ:inhomoerror}
 	h\|\phi-\phi^h\|_{H^{1}(\Omega)} + \|\phi-\phi^h\|_{L^2(\Omega)}\lesssim h^{k+1}
 	(\|\phi\|_{H^{k+1}(\Omega)}+\|g\|_{H^{k+\frac{1}{2}}(\partial\Omega)}).
 \end{equation}
\end{theorem}
\begin{proof}
  The error estimate \eqref{equ:inhomoerror} can be established by the optimal error estimate for homogeneous problems \eqref{equ:homoerror} and the standard lift argument \cite{Ci2002, BS2008, StFi2008}.
\end{proof}

\section{Main results}
\label{sec:main}
Suppose the domain $\Omega$ is convex. The main result of this paper is the following theorem.
\begin{theorem} \label{thm:maxerror}
Let $u$ be the solution of \eqref{equ:varform} and $u^h$ be its virtual element  solution.  If $u\in W^{k+1,\infty}(\Omega)$ and $f\in H^k(\Omega)$, then the following result holds
    \begin{equation}\label{equ:h1max}
      \|\nabla u-\nabla \Pi^{\nabla}_k u_{h}\|_{L^{\infty}( \Omega)} \lesssim h^{k}|\ln h|^{2}(\|u\|_{W^{k+1,\infty}(\Omega)}+\|f\|_{H^{k}(\Omega)}).
    \end{equation}
   Furthermore, if we assume $k\ge 2$, then we have 
   \begin{equation}   \label{equ:l2max}
     \|u-u^h\|_{L^{\infty} (\Omega)} \lesssim h^{k+1}|\ln h|^{4}(\|u\|_{W^{k+1,\infty}(\Omega)}+\|f\|_{H^{k}(\Omega)}).
    \end{equation}
\end{theorem}

\begin{remark}
For the classical finite element method, Schatz and Wahlbin \cite{ScWa1978} proved the following maximum norm error estimates 
\begin{equation}\label{equ:h1maxs0}
      \|u-u^{h}\|_{W^{1,\infty}( \Omega)} \lesssim h^{k}\|u\|_{W^{k+1,\infty}(\Omega)},
    \end{equation}
   and 
   \begin{equation}   \label{equ:l2maxs1}
     \|u-u^h\|_{L^{\infty} (\Omega)} \lesssim h^{k+1}|\ln h|^{\overline{k}}\|u\|_{W^{k+1,\infty}(\Omega)},
    \end{equation}\
 where  $u_h$ is the classical continuous finite element solution of degree $k$ and  $\overline{k}=1$ if $k=1$ and  $\overline{k}=0$ if $k\geq 1$.  In contrast the maximal norm error estimates of  the classical finite element methods, the maximal norm error estimates for virtual element methods have a high-order power of $|\ln h|$. 
    \end{remark}.

In the rest of  this section, we shall present a proof of our main result. 
\subsection{Local regularity result}
For any point $\mathbf{x}$ in $\Omega$,  let $G_{\mathbf{x}}\in W^{1,1}(\Omega)$ be the standard Green's function  for \eqref{equ:model} which is defined as 
\begin{equation}
	a(v, G_{\mathbf{x}}) = v(\mathbf{x}), \quad \forall v \in H^1_{0}(\Omega). 
\end{equation}

Let  $M$ be the set of corner points(or vertices) of the domain $\Omega$.  Let 
\begin{equation}\label{equ:cornerset}
  B(M, r) = \Omega\cap \bigcup_{\mathbf{y}\in M} B(\mathbf{y},r), 
\end{equation}
where $B(\mathbf{y},r)$ is the disk with  radius $r$ centered at $\mathbf{y}$.  For any $E_0\in\mathcal{T}_h$ and $r>0$, define 
\begin{equation}\label{equ:eleset}
	B_{E_0,r} = \{\mathbf{y}\in\Omega|~\rho(\mathbf{y},E_0)\le r\} \cup B(M,r),
\end{equation}
where the distance function $\rho(\cdot, \cdot)$ is defined as 
\begin{equation}\label{equ:dist}
\rho(\mathbf{x},E_0) = \inf_{\mathbf{y}\in E_0} |\mathbf{x}-\mathbf{y}|. 
\end{equation}
Similarly, for any $\mathbf{x}\in\Omega$, let 
\begin{equation}\label{equ:pointset}
	B_{\mathbf{x},r} =  \left( \Omega\cap B(\mathbf{x},r) \right) \cup B(M,r).
\end{equation}

In this subsection,  we shall establish some local regularity result.
Under the assumption that $\Omega$ is smooth enough such as ball,    \cite{Kras1967}     showed,
for any $r>0$, there holds
\begin{align}
\|G_\mathbf{x}\|_{W^{s,\infty}(\Omega\backslash B(\mathbf{x},r))}\lesssim r^{-s}.\label{addnew}
\end{align}
 In this paper, we prove a more generalized results. The estimate of Green's function is one the key ingredient in  the establishment 
 local regularity result \eqref{equ:lemma2.5ntt00}  which will be used to prove \eqref{equ:l2local}.

We start with the following lemma. 
\begin{lemma}
\label{Lem:Result2}
Assume that $s\geq 3$ is a positive integer. Then there holds
\begin{equation}\label{equ:lemma2.5n}
        \|G_\mathbf{x}\|_{H^{s}(\Omega\backslash B_{\mathbf{x},r})}\lesssim r^{1-s}|\ln r|,
\end{equation}
and for any $s\geq 2$, there holds
\begin{equation}\label{equ:lemma2.5nnn}
\|G_\mathbf{x}\|_{W^{s,\infty}(\Omega\backslash B_{\mathbf{x},r})}\lesssim r^{-s}|\ln r|.
\end{equation}
\end{lemma}
\begin{proof}
	We firstly prove \eqref{equ:lemma2.5n}.  Let $\nu_{1}(\mathbf{z})$ be the cutoff function such that $\nu_{1}(\mathbf{z})=1$ if
$\mathbf{z}\in \Omega\backslash B_{\mathbf{x},r}$, and $\nu_{1}(\mathbf{z})=0$ if $z\in
B_{\mathbf{x},\frac{r}{2}}$. It is easy to see that $\|\nu_{1}\|_{W^{k,\infty}(\Omega)}\lesssim r^{-k}$ for any positive integer $k$.

Let
\begin{equation}\label{equ:essss}
	\overline{G}_{\mathbf{x}}(z)=\nu_{1}(\mathbf{z})G_{\mathbf{x}}(\mathbf{z}).
\end{equation}
We firstly estimate $\|\overline{G}_{\mathbf{x}}\|_{H^{2}(\Omega)}$. To do this, we observe that,  for any $\mathbf{z}\in  \Omega\backslash B_{\mathbf{x},r}$, there holds
\begin{equation}\label{equ:essss1}
G_{\mathbf{x}}(\mathbf{z})=\overline{G}_{\mathbf{x}}(\mathbf{z}).
\end{equation}
Furthermore,  $\L \overline{G}_{\mathbf{x}}(\mathbf{z})=0$ if $\mathbf{z}\in B_{\mathbf{x},\frac{r}{2}}\cup(\Omega\backslash B_{\mathbf{x},r})$.  From \eqref{equ:essss1}, we can deduce that  
\begin{equation}\label{equ:essss2}
	\begin{split}
		&\|\overline{G}_{\mathbf{x}}\|_{H^{2}(\Omega)}
		\lesssim \|\L\overline{G}_{\mathbf{x}}\|_{L^{2}(\Omega)}
		=\|\L\overline{G}_{\mathbf{x}}\|_{L^{2}(B_{\mathbf{x},r}\backslash B_{\mathbf{x},\frac{r}{2}})} \\
		\lesssim &\|G_{\mathbf{x}}\|_{H^{2}(B_{\mathbf{x},r}\backslash B_{\mathbf{x},\frac{r}{2}})}
+r^{-1}\|G_{\mathbf{x}}\|_{H^{1}(B_{\mathbf{x},r}\backslash B_{\mathbf{x},\frac{r}{2}})}+r^{-2}
\|G_{\mathbf{x}}\|_{L^{2}(B_{\mathbf{x},r}\backslash B_{\mathbf{x},\frac{r}{2}})}\\
\lesssim &\|G_{\mathbf{x}}\|_{H^{2}(B_{\mathbf{x},r}\backslash B_{\mathbf{x},\frac{r}{2}})}
+r^{-1}\|G_{\mathbf{x}}\|_{H^{1}(B_{\mathbf{x},r} \backslash B_{\mathbf{x},\frac{r}{2}})}+r^{-1}
\|G_{\mathbf{x}}\|_{L^{\infty}(B_{\mathbf{x},r} \backslash B_{\mathbf{x},\frac{r}{2}})}\\
\lesssim  &r^{-1}+r^{-1}+r^{-1}|\ln r| \\
\lesssim &r^{-1}|\ln r|,
	\end{split}
\end{equation}
where we have used the following estimate (see \cite[Lemma 2.1]{HeZZ2018})  
\begin{equation}\label{equ:greenreg}
	\|G_{\mathbf{x}}\|_{H^{2}(\Omega\backslash B_{\mathbf{x},r})}+r^{-1}\|G_{\mathbf{x}}\|_{H^{1}(B_{\mathbf{x},r}\backslash B_{\mathbf{x},\frac{r}{2}})}
+r^{-1}|\ln r|^{-1}\|G_{\mathbf{x}}\|_{L^{\infty}(\Omega\backslash B_{\mathbf{x},\frac{r}{2}})}
\lesssim r^{-1}.
\end{equation}

In the following, we shall use \eqref{equ:essss2} and \eqref{equ:greenreg} to estimate  $\|G_{\mathbf{x}}\|_{H^{s}(\Omega\backslash B_{\mathbf{x},r})}$ for $s\geq 3$.
Assume that $\theta(\mathbf{x})=\rho(\mathbf{x},M)$ where $M$ is the set of corner points for the domain $\Omega$.  We defined the weighted norm $\|\cdot \|_{\kappa_{\alpha_{1}}^{\alpha_{2}}(\Omega)}$  as
\begin{equation}\label{equ:eh000}
	\|\omega\|_{\kappa_{\alpha_{1}}^{\alpha_{2}}(\Omega)}
=\sum\limits_{|\mathbf{\beta}|\leq \alpha_{2}}\|\theta^{|\mathbf{\beta}|-\alpha_{1}}\partial^{\mathbf{\beta}}\omega\|_{L^{2}(\Omega)}. 
\end{equation}
We start with the estimation of  $\|\overline{G}_{\mathbf{x}}\|_{\kappa_{1}^{0}(\Omega)}$ and $\|\L\overline{G}_{\mathbf{x}}\|_{H^{1}(B_{\mathbf{x},r}\backslash B_{\mathbf{x},\frac{r}{2}})}$. 
Let $r_{0}=\frac{r}{2}$. Note that $\overline{G}_{\mathbf{x}}(\mathbf{z})=0$ for all $\mathbf{z}\in B_{\mathbf{x},\frac{r}{2}}$. 
Assume that $s\geq 3$. By \eqref{equ:greenreg}, we have
\begin{equation}
	\begin{split}
		\|\overline{G}_{\mathbf{x}}\|_{\kappa_{s}^{0}(\Omega)}^{2}
&=\|\overline{G}_{\mathbf{x}}\|_{\kappa_{s}^{0}(B(M,r_{0}))}^{2}+
\|\overline{G}_{\mathbf{x}}\|_{\kappa_{s}^{0}(\Omega\backslash B(M,r_{0}))}^{2} \\
&=\|\overline{G}_{\mathbf{x}}\|_{\kappa_{s}^{0}(\Omega\backslash B(M,r_{0}))}^{2}
=\int_{\Omega\backslash B(M,r_{0})}\theta^{-2s}(\mathbf{z})\overline{G}_{\mathbf{x}}^{2}(\mathbf{z})d\mathbf{z}\\
&\lesssim \|\theta^{-2s}\|_{L^{1}(\Omega\backslash B(M,r_{0}))}
\|\overline{G}_{\mathbf{x}}^{2}\|_{L^{\infty}(\Omega)}\\
&\lesssim r_{0}^{2-2s}\|G_{\mathbf{x}}^{2}\|_{L^{\infty}(\Omega\backslash B_{\mathbf{x},\frac{r}{2}})}\\
& \lesssim r^{2-2s}|\ln r|^{2}.
	\end{split}
\end{equation}
It  implies 
\begin{equation}\label{equ:epp2}
\|\overline{G}_{\mathbf{x}}\|_{\kappa_{s}^{0}(\Omega)}\lesssim r^{1-s}|\ln r|.
\end{equation}
Next,  we estimate $\|\L\overline{G}_{\mathbf{x}}\|_{H^{1}(\Omega)}$. We  observes that,
for any $\mathbf{z}\in  \Omega\backslash B_{\mathbf{x}},\frac{r}{2}$, there holds
\begin{equation}\label{equ:epp401}
	\begin{split}
&\L\overline{G}_{\mathbf{x}}(\mathbf{z})\\=&
-\nabla\cdot(\nu_{1}(\mathbf{z})\alpha(\mathbf{z})\nabla G_{\mathbf{x}}(\mathbf{z}))
-\nabla\cdot(G_{\mathbf{x}}(\mathbf{z})\alpha(\mathbf{z})\nabla \nu_{1}(\mathbf{z}))
\\=&-\nabla \nu_{1}(\mathbf{z}) \cdot \alpha(\mathbf{z})\nabla G_{\mathbf{x}}(\mathbf{z})
-\nu_{1}(\mathbf{z})\nabla\cdot(\alpha(\mathbf{z})\nabla G_{\mathbf{x}}(\mathbf{z}))
-\nabla\cdot(G_{\mathbf{x}}(\mathbf{z})\alpha(\mathbf{z})\nabla \nu_{1}(\mathbf{z}))
\\=&-\nabla \nu_{1}(\mathbf{z}) \cdot  \alpha(\mathbf{z})\nabla G_{\mathbf{x}}(\mathbf{z})
-\nabla\cdot(G_{\mathbf{x}}(\mathbf{z})\alpha(\mathbf{z})\nabla \nu_{1}(\mathbf{z}))
\\=&-\nabla \nu_{1}(\mathbf{z}) \cdot \alpha(\mathbf{z})\nabla G_{\mathbf{x}}(\mathbf{z})
-G_{\mathbf{x}}(\mathbf{z})\nabla\cdot(\alpha(\mathbf{z})\nabla \nu_{1}(\mathbf{z}))
-\nabla G_{\mathbf{x}}(\mathbf{z})\cdot\alpha(\mathbf{z})\nabla \nu_{1}(\mathbf{z})).
	\end{split}
\end{equation}
By \eqref{equ:epp401},  for any $z\in B_{\mathbf{x},r} \backslash B_{\mathbf{x},\frac{r}{2}}$, there holds
\begin{equation}\label{equ:myerror}
	\begin{split}
		&\frac{\partial \L\overline{G}_{\mathbf{x}}(\mathbf{z})}{\partial z_{l}} \\
		=&-\frac{\partial  \nabla \nu_{1}(\mathbf{z})}{\partial z_{l}}
		\cdot \alpha(\mathbf{z})\nabla G_{\mathbf{x}}(\mathbf{z})-
	\nabla \nu_{1}(\mathbf{z}) \cdot \frac{\partial(\alpha(\mathbf{z})\nabla G_{\mathbf{x}}(\mathbf{z}))}
	{\partial z_{l}}\\&
	-\frac{\partial G_{\mathbf{x}}(\mathbf{z})}{\partial z_{l}}
	\nabla\cdot(\alpha(\mathbf{z})\nabla \nu_{1}(\mathbf{z}))
	-G_{\mathbf{x}}(\mathbf{z})\frac{\partial \nabla\cdot(\alpha(\mathbf{z})\nabla \nu_{1}(\mathbf{z}))
	}{\partial z_{l}}\\&
	-\frac{\partial \nabla G_{\mathbf{x}}(\mathbf{z})}{\partial z_{l}}\cdot 
	\alpha(\mathbf{z})\nabla \nu_{1}(\mathbf{z})
	-\nabla G_{\mathbf{x}}(\mathbf{z})\cdot\frac{\partial(\alpha(\mathbf{z})\nabla \nu_{1}(\mathbf{z}))
	}{\partial z_{l}}.
	\end{split}
\end{equation}
Combining the estimates \eqref{equ:greenreg} and \eqref{equ:myerror}, we can deduce that 
\begin{equation}\label{equ:epp4}
	\begin{split}
		&\|\L\overline{G}_{\mathbf{x}}\|_{H^{1}(\Omega)}=\|\L\overline{G}_{\mathbf{x}}\|_{H^{1}(B_{\mathbf{x},r}\backslash B_{\mathbf{x},\frac{r}{2}})}\\
		\lesssim & \|\nu_{1}\|_{W^{3,\infty}(B_{\mathbf{x},r}\backslash B_{x,\frac{r}{2}})}
\|G_{\mathbf{x}}\|_{L^{2}(B_{\mathbf{x},r}\backslash B_{\mathbf{x},\frac{r}{2}})}+ \\&
		\|\nu_{1}\|_{W^{2,\infty}(B_{\mathbf{x},r}\backslash B_{\mathbf{x},\frac{r}{2}})}
\|G_{x}\|_{H^{1}(B_{\mathbf{x},r}\backslash B_{\mathbf{x},\frac{r}{2}})} + \\
&\|\nu_{1}\|_{W^{1,\infty}(B_{\mathbf{x},r}\backslash B_{\mathbf{x},\frac{r}{2}})}
\|G_{\mathbf{x}}\|_{H^{2}(B_{\mathbf{x},r}\backslash B_{\mathbf{x},\frac{r}{2}})}\\
\lesssim  &r^{-3}r|\ln r|+r^{-2}|\ln r|+r^{-1}r^{-}|\ln r|\\
\lesssim &r^{-2}|\ln r|.
	\end{split}
\end{equation}

Note that $\overline{G}_{\mathbf{x}}(\mathbf{z})=0$ for all $\mathbf{z}\in \partial\Omega$.
By Theorem 3.7 in \cite{BNZ2005}, \eqref{equ:epp2},  and  \eqref{equ:epp4}, we can derive 
that 
\begin{equation}\label{equ:OSOS}
	\|\overline{G}_{\mathbf{x}}\|_{H^{3}(\Omega)}\lesssim
(\|\overline{G}_{\mathbf{x}}\|_{\kappa_{3}^{0}(\Omega)}
+\|\L\overline{G}_{\mathbf{x}}\|_{H^{1}(\Omega)})
\lesssim r^{-2}|\ln r|,
\end{equation}
which implies 
\begin{equation}
	\|G_\mathbf{x}\|_{H^3(\Omega\backslash B_{\mathbf{x},r})} = \|\overline{G}_{\mathbf{x}}\|_{H^3(\Omega\backslash B_{\mathbf{x},r})} \le \|\overline{G}_{\mathbf{x}}\|_{H^{3}(\Omega)}\lesssim r^{-2}|\ln r.
\end{equation}

 When $s\ge 4$, we shall prove it by induction on $s$.  Assume that, for any $1\leq l\leq s$,   there holds
 \begin{align}\label{equ:ehpp1}
 \|G_{\mathbf{x}}\|_{H^{s-l}(\Omega\backslash B_{\mathbf{x},r})}\lesssim r^{1+l-s}|\ln r|.
\end{align}
Similarly to \eqref{equ:myerror} and \eqref{equ:epp4}, we have 
\begin{equation}\label{equ:myerror2}
	\begin{split}
		&\|\L\overline{G}_{\mathbf{x}}\|_{H^{s-2}(\Omega)}=\|\L\overline{G}_{\mathbf{x}}\|_{H^{s-2}(B_{\mathbf{x},r}\backslash B_{\mathbf{x},\frac{r}{2}})}\\
		\lesssim &\sum\limits_{l=1}^{s}
\|\nu_{1}\|_{W^{l,\infty}(B_{\mathbf{x},r}\backslash B_{\mathbf{x},\frac{r}{2}})}
\|G_{\mathbf{x}}\|_{H^{s-l}(B_{\mathbf{x},r}\backslash B_{\mathbf{x},\frac{r}{2}})}\\
\lesssim  &\sum\limits_{l=1}^{s}r^{-l}cr^{1+l-s}|\ln r|\\
&\lesssim r^{-1}|\ln r|.
	\end{split}
\end{equation}
Again,  Theorem 3.7 in \cite{BNZ2005}, \eqref{equ:epp2},  and  \eqref{equ:epp4} give 
\begin{equation*}
\|\overline{G}_{\mathbf{x}}\|_{H^{s}(\Omega)}
\lesssim
(\|\overline{G}_{\mathbf{x}}\|_{\kappa_{s}^{0}(\Omega)}
+\|\L\overline{G}_{\mathbf{x}}\|_{H^{s-2}(\Omega)})
\lesssim r^{1-s}|\ln r|+cr^{1-s}|\ln r|
\lesssim r^{1-s}|\ln r|,
 \end{equation*}
which implies
\begin{equation}\label{equ:eh0000}
 \|G_{\mathbf{x}}\|_{H^{s}(\Omega\backslash B_{\mathbf{x},r})}=\|\overline{G}_{\mathbf{x}}\|_{H^{s}(\Omega\backslash B_{\mathbf{x},r})}\leq
 \|\overline{G}_{\mathbf{x}}\|_{H^{s}(\Omega)}
 \lesssim r^{1-s}|\ln r|.
\end{equation}

Now, we turn to  prove \eqref{equ:lemma2.5nnn}.
Let $\Phi=\{\frac{\mathbf{y}-\mathbf{x}}{r}|~\mathbf{y}\in \Omega\}$ and $\Phi_{1}=\{\frac{\mathbf{y}-\mathbf{x}}{r}|~\mathbf{y}\in B_{\mathbf{x},r}\}$. We introduce $\psi(t)$ by letting 
\begin{equation}\label{equ:first001}
\psi(\mathbf{z})=G_{\mathbf{x}}(r\mathbf{z}),\qquad \forall \mathbf{z}\in \Phi.
\end{equation}
 For any $s\geq 3$,  \eqref{equ:lemma2.5n} and \eqref{equ:first001} give 
\begin{equation*}
\|\psi\|_{H^{s}(\Phi\backslash \Phi_{1})}=r^{s-1}\|G_{x}\|_{H^{s}(\Omega\backslash B_{\mathbf{x},r})}
\lesssim |\ln r|.
\end{equation*}
For any $s_{1}\geq 2$, it implies that
\begin{equation}\label{equ:first002}
\|\psi\|_{W^{s_{1},\infty}(\Phi\backslash \Phi_{1})}\leq c|\ln r|.
\end{equation}
From \eqref{equ:first002}, we can derive that 
\begin{equation*}
\|G_{\mathbf{x}}\|_{W^{s_{1},\infty}(\Omega\backslash B_{\mathbf{x},r})}
\leq cr^{-s_{1}}\|\psi\|_{W^{s_{1},\infty}(\Phi\backslash \Phi_{1})}
\leq cr^{-s_{1}}|\ln r|,
\end{equation*}
which concludes the proof of \eqref{equ:lemma2.5nnn}.
\end{proof}

 Using the above Lemma, we shall prove the following point-wise estimate.
 
 \begin{lemma}
\label{lem:highgreen}
 Let $\gamma>2$,  $\overline{\gamma}=\gamma-1$, and  $B_{E_0,r}$ be defined as in \eqref{equ:eleset}.
Assume  that  $\mathbf{y}\in \Omega\backslash B_{E_0,\gamma r}$ and
 $|\beta|\geq 2$ is a positive integer.
 Then
\begin{equation}\label{lemma2.5nnnn000}
\left\|\left.\frac{\partial^{\beta}G_{\mathbf{z}}}{\partial
z_{1}^{\beta_{1}}\partial
z_{2}^{\beta_{2}}}\right|_{\mathbf{z}=\mathbf{y}}\right\|_{L^{\infty}(B_{E_0,r})}\lesssim
  \overline{\gamma}^{-|\beta|}r^{-|\beta|}|\ln r|.
\end{equation}
\end{lemma}
\begin{proof}
Assume that $\mathbf{x}\in \Omega$. Notice that
\begin{equation}\label{equ:uuuvvv}
G_{\mathbf{y}}(\mathbf{x})=G_{\mathbf{x}}(\mathbf{y}) \quad \text{and} \quad G_{\mathbf{y}+\Delta
\mathbf{y}}(\mathbf{x})=G_{\mathbf{x}}(\mathbf{y}+\Delta \mathbf{y}).
\end{equation}
By setting $|\Delta \mathbf{y}|\to 0$ in  \eqref{equ:uuuvvv}, we can deduce that 
$\left.\frac{\partial G_{\mathbf{z}}(\mathbf{x})}{\partial z_{i}}\right|_{\mathbf{z}=\mathbf{y}}=
\frac{\partial G_{\mathbf{x}}(\mathbf{y})}{\partial y_{i}}$. Furthermore, we have
\begin{equation}\label{equ:uuuuuu3}
\left.\frac{\partial^{\beta}G_{\mathbf{z}}(\mathbf{x})}{\partial
z_{1}^{\beta_{1}}\partial
z_{2}^{\beta_{2}}}\right|_{\mathbf{z}=\mathbf{y}}=\frac{\partial^{\beta}G_{\mathbf{x}}(\mathbf{y})}{\partial
y_{1}^{\beta_{1}}\partial y_{2}^{\beta_{2}}}.
\end{equation}
Since  $\mathbf{y}\in \Omega\backslash B_{E_0, \gamma r}$, 
it follows that  for any $\mathbf{x}\in B_{E_0,r}$, there holds
\begin{equation} \label{equ:fffooo}
 |\mathbf{x}-\mathbf{y}|\geq \rho(\mathbf{y},E_{0})-\rho(\mathbf{x},E_{0})=\gamma r-r=\overline{\gamma} r,
 \end{equation}
 and 
 \begin{equation}\label{equ:fffttt11}
 \rho(\mathbf{y},M)\geq \gamma r>\overline{\gamma}r.
 \end{equation}
From \eqref{equ:fffooo} and \eqref{equ:fffttt11}, we can show
\begin{equation}\label{equ:zengjia1}
\mathbf{y}\in \Omega\backslash B_{\mathbf{x},\overline{\gamma}r}.
\end{equation}
It implies that 
\begin{equation}\label{equ:zengjia2}
\begin{split}
\left\|\left.\frac{\partial^{\beta}G_{\mathbf{z}}}{\partial
z_{1}^{\beta_{1}}\partial
z_{2}^{\beta_{2}}}\right|_{\mathbf{z}=\mathbf{y}}\right\|_{L^{\infty}(B_{E_0, r})}
&\leq
\max_{\mathbf{x}\in B_{E_{0}, r}}
|\nabla^{|\beta|}G_{\mathbf{x}}|_{L^{\infty}(\Omega\backslash B_{\mathbf{x},\overline{\gamma}r})}
\\&
\lesssim
 \overline{\gamma}^{-|\beta|}r^{-|\beta|}|\ln r|,
\end{split}
\end{equation}
where we have used  \eqref{equ:lemma2.5nnn}.
\end{proof}
\vskip 0.3cm

Based on Lemma \ref{lem:highgreen},  we establish the following  local regularity result.

\begin{theorem}\label{thm:reg}
Let $r\ge d>0$ and $\gamma>2$.
Assume that $\nu\in L^{2}(\Omega)$ satisfies 
$\nu(\mathbf{x})=0$ if $\mathbf{x}\in \Omega\backslash B_{E_{0},d}$ and $\chi(\mathbf{x})\in H_{0}^{1}(\Omega)$
    is the solution to   $\L\chi(x):\equiv -\nabla\cdot(\alpha\nabla \chi(\mathbf{x}))=\nu(\mathbf{x})$.
     Then there holds
\begin{equation}\label{equ:lemma2.5ntt00}
        \|\chi\|_{H^{3}(\Omega\backslash B_{E_0,\gamma d})}\lesssim \gamma^{-1}d^{-1}|\ln d|\|\nu\|_{L^{2}(\Omega)}.
    \end{equation}
\end{theorem}
\begin{proof}
	Let $r\geq d$. Assume that $\mathbf{y}\in \Omega\backslash B_{E_0,\gamma r}$. Note that $\nu(\mathbf{x})=0$ if $\mathbf{x}\in \Omega\backslash B_{E_0,r}$. Then, Lemma \ref{lem:highgreen} implies 
	\begin{equation}\label{equ:pp000}
		\begin{split}
			\left|\frac{\partial^{\beta}\chi(\mathbf{y})}{\partial y_{1}^{\beta_{1}}\partial y_{2}^{\beta_{2}}}\right|\le
			&\int_{\Omega}\left|\left.\frac{\partial^{\beta}G_{\mathbf{z}}(\mathbf{x})}{\partial
z_{1}^{\beta_{1}}\partial z_{2}^{\beta_{2}}}\right|_{z=y}\right||\nu(\mathbf{x})|d\mathbf{x}
=\int_{B_{E_0,r}}\left|\left.\frac{\partial^{\beta}G_{\mathbf{z}}(\mathbf{x})}{\partial
z_{1}^{\beta_{1}}\partial z_{2}^{\beta_{2}}}\right|_{\mathbf{z}=\mathbf{y}}\right||\nu(\mathbf{x})|d\mathbf{x} \\
\lesssim &\left\|\left.\frac{\partial^{\beta}G_{\mathbf{z}}(\mathbf{x})}{\partial
z_{1}^{\beta_{1}}\partial z_{2}^{\beta_{2}}}\right|_{\mathbf{z}=\mathbf{y}}\right\|_{L^{2}(B_{E_{0}, r})}
\|\nu\|_{L^{2}(\Omega)}\\
\lesssim & (\gamma-1)^{1-\beta}r^{1-\beta}|\ln r|\|\nu\|_{L^{2}(\Omega)}.
		\end{split}
	\end{equation}
Let $|\beta|=3$. By \eqref{equ:pp000}, we have
\begin{align}\label{equ:elll}
        \|\chi\|_{W^{3,\infty}(\Omega\backslash B_{E_{0},\gamma r})}\lesssim (\gamma-1)^{-2}r^{-2}|\ln r|\|\nu\|_{L^{2}(\Omega)}.
    \end{align}
Let $d_{i}=2^{i-1}d$ for any positive integer $i$. We also assume that $i_{0}$
is the  minimal positive integer satisfying $\Omega\subset B_{E_0,\gamma d_{i_{0}+1}}$. Note that $\gamma\geq 2$. 
 By \eqref{equ:elll}, we have
\begin{equation}
	\begin{split}
        \|\chi\|_{H^{3}(\Omega\backslash B_{E_{0}, \gamma d})}&=
        \sqrt{\sum\limits_{i=1}^{i_{0}}\|\chi\|_{H^{3}(B_{E_{0}, \gamma d_{i+1}}\backslash B_{E_{0}, \gamma d_{i}})}^{2}}\\
        &\leq
        \sqrt{\sum\limits_{i=1}^{i_{0}}c\gamma^{2}d_{i}^{2}\|\chi\|_{W^{3,\infty}(B_{E_{0},\gamma d_{i+1}}\backslash B_{E_{0},\gamma d_{i}})}^{2}}
        \\& \leq  \sqrt{\sum\limits_{i=1}^{i_{0}}c\gamma^{2}d_{i}^{2}c(\gamma-1)^{-4}d_{i}^{-4}|\ln d_{i}|^{2}}\|\nu\|_{L^{2}(\Omega)} \\
&\lesssim \gamma (\gamma-1)^{-2}\|\nu\|_{L^{2}(\Omega)}\sqrt{\sum\limits_{i=1}^{i_{0}}d_{i}^{-2}|\ln d_{i}|^{2}}
        \\& \lesssim \gamma (\gamma-1)^{-2}d^{-1}|\ln d|\|\nu\|_{L^{2}(\Omega)}\\
        &\lesssim \gamma^{-1}d^{-1}|\ln d|\|\nu\|_{L^{2}(\Omega)}.
    \end{split}
\end{equation}
which completes the proof of  \eqref{equ:lemma2.5ntt00}.
\end{proof}

%
%

\subsection{Cutoff functions}
\label{ssec:cut}
Let $E_0$ be an arbitrary element in $\mathcal{T}_h$.  
Assume that $l_{0}$ is a positive integer satisfying
 \begin{equation}\label{equ:l0define}
 (l_{0}-1)2^{l_{0}-1}
 \leq
 \max\limits_{\mathbf{x}\in \Omega}\rho(\mathbf{x},E_{0})
 \leq l_{0}2^{l_{0}}.
 \end{equation}
  From \eqref{equ:l0define} and the regularity of the mesh, it not hard to deduce that 
 \begin{equation}\label{equ:l0estimate}
l_{0}\lesssim |\ln h|.
\end{equation}
Let $\hat{B}_{E_{0},r}$ be defined by
\begin{align}
\hat{B}_{E_{0},r}=\{\mathbf{x}|~\rho(\mathbf{x},E_{0})\leq r\}\cup \{\mathbf{x}|~\rho(\mathbf{x},M)\leq r\},\label{hatrK0}
\end{align}
with $M$ being the  set of corner points(or vertices) of the domain $\Omega$.

For $0\leq l\leq l_{0}$,  let
\begin{align}
r_{l}=2^{l}lh.\label{rldefine}
\end{align}
 Let  $\hat{B}_{E_{0},r_{-1}}=\emptyset$.
Assume that $\phi_{l}(\mathbf{x})~(0\leq l\leq l_{0})\in C^{\infty}(\mathbb{R}^{2})$
satisfies,   $0\leq \phi_{l}(\mathbf{x})\leq 1$ for all $\mathbf{x}\in \mathbb{R}^{2}$,
and $\phi_{l}(\mathbf{x})=1$ if $\mathbf{x}\in \hat{B}_{E_{0},r_{l}}$,
and  $ \phi_{l}(\mathbf{x})=0$ if $\mathbf{x}\in \mathbb{R}^{2}\backslash \hat{B}_{E_{0},r_{l+1}}$.
We also assume that there exists a constant $c$ independent of $\mathbf{x}$ and $l$ such that
\begin{align}
 |\nabla^{m}\phi_{l}(\mathbf{x})|\leq cr_{l}^{-m},\qquad \forall  0\leq l\leq l_{0}.\label{newnew0000}
 \end{align}
 Set 
  \begin{equation}\label{equ:poly0000}
\hat{\phi}_{l}(\mathbf{x})=
\begin{cases}
\phi_{l}(\mathbf{x}),\quad  \text{if}\quad  l=0,\\
 \phi_{l}(\mathbf{x})-\phi_{l-1}(\mathbf{x}),\quad  \text{if}\quad  1\leq l\leq l_{0}-1.
\end{cases}
 \end{equation}
 Notice that \eqref{equ:l0define} implies  $\Omega \subset \hat{B}_{E_0, r_{l_0}}$. 
By \eqref{newnew0000}  and \eqref{equ:poly0000}, we have, if $\mathbf{x}\in \Omega$, then  
\begin{align}
  \sum\limits_{l=0}^{l_0}\hat{\phi}_{l}(\mathbf{x})=\phi_{l_{0}}(\mathbf{x})
  =1,\label{newnew000}
\end{align}
 and 
\begin{align}
  |\nabla^{m}\hat{\phi}_{l}(\mathbf{x})|\leq
  |\nabla^{m}\phi_{l}(\mathbf{x})|+|\nabla^{m}\phi_{l-1}(\mathbf{x})|\lesssim
  r_{l}^{-m}+r_{l-1}^{-m}\lesssim
  r_{l}^{-m}.\label{newnew00}
 \end{align}
 Using the above cutoff functions, we define a partition of the exact solution $u$. 
Let $\mathbf{x}_{0}\in E_{0}$ satisfy 
$|\nabla^{m}u(\mathbf{x}_{0})|\leq \|u\|_{W^{m,\infty}(\Omega)}~(0\leq m\leq k+1)$
. Assume that $\hat{u}(\mathbf{x})$ is a global  polynomial of degree $k$ on $\Omega$  satisfying
 \begin{equation}\label{equ:poly00}
  	\nabla^{m}\hat{u}(\mathbf{x}_{0})=\nabla^{m}u(\mathbf{x}_{0}),\qquad \forall  0\leq m\leq k.
  	\end{equation}
  The existence of $\hat{u}(\mathbf{x})$ can be guaranteed using the averaged Taylor polynomial in \cite{BS2008}.
  By \eqref{equ:poly00}, we have, for any $r>0$, there holds  \cite{BS2008}
  	\begin{equation}\label{equ:poly}
  \|u-\hat{u}\|_{W^{m,\infty}(\Omega\cap\hat{B}_{E_{0},r})}\lesssim r^{k+1-m}\|u\|_{W^{k+1,\infty}(\Omega)},\qquad \forall  0\leq m\leq k.
  \end{equation}
 
  Then, we  can define the following partition of the exact solution
\begin{equation}\label{equ:partition}
u(\mathbf{x})=\hat{u}(\mathbf{x})+\sum\limits_{l=0}^{l_{0}}u_{l}(\mathbf{x}),
\end{equation}
where
\begin{equation}\label{equ:partdef}
u_{l}(\mathbf{x})=\hat{\phi}_{l}(\mathbf{x})[u(\mathbf{x})-\hat{u}(\mathbf{x})].
\end{equation}
We can define a similar partition for the right hand side function as 
\begin{equation}
	f(\mathbf{x})=\widehat{f}(\mathbf{x})+\sum\limits_{l=0}^{l_{0}}f_{l}(\mathbf{x}),
\end{equation}
such that 
\begin{equation}
	a(\hat{u}, v) = (\hat{f}, v),  \quad  \forall v\in H^1_0(\Omega),
\end{equation}
and 
\begin{equation}\label{equ:subprob}
	a(u_l, v) = (f_l, v), \quad  \forall v\in H^1_0(\Omega),
\end{equation}
for $l= 0, \ldots, l_0$.

Once we obtain the decomposition of $f$,  we define the virtual element solution $\hat{u}^h\in V_{h}$  of $\hat{u}$ as 
\begin{equation}\label{equ:vem_sol_uhat}
\begin{cases}
	a_h(\hat{u}^h, v^h) = (\hat{f}, \Pi_{k-1}^{0} v^{h}), \quad \forall v^h\in V_{h,0},\\
	\hat{u}^h(\mathbf{x})=\hat{u}(\mathbf{x}),\qquad  \forall \mathbf{x}\in \partial\Omega,
\end{cases}
\end{equation}
and the virtual element solution  $u^h_l\in V_{h}$ of $u_{l}$ as 
\begin{equation}\label{equ:vem_sol_ul}
\begin{cases}
a_h(u_l^h, v^h) = (f_l, \Pi_{k-1}^{0} v^{h}), \quad \forall v^h\in V_{h,0},\\
\hat{u}_{l}^h(\mathbf{x})=I_{h}\hat{u}_l(\mathbf{x}),\qquad  \forall \mathbf{x}\in \partial\Omega,
\end{cases}
\end{equation}
where $I_h$ is the interpolation operator of the virtual element space  $V_h$ defined in \eqref{equ:interpolation}.

For the global polynomial, we can show the following equivalence relationship:

\begin{lemma}\label{lem:equal}
Let $\hat{u}(\mathbf{x})$ be the  polynomial defined in \eqref{equ:poly} and $\hat{u}^h(\mathbf{x})$ be its virtual element solution defined in \eqref{equ:vem_sol_uhat}. Then, we have
\begin{equation}\label{equ:projequal}
	\hat{u}^h(\mathbf{x}) = \hat{u}(\mathbf{x}). 
\end{equation}
\end{lemma}
\begin{proof}
	Since $\hat{u}$ is a polynomial of degree $k$, it follows that 
	\begin{equation}\label{equ:polynomial}
		a_h(\hat{u}, v^h) = a(\hat{u}, v^h), \quad \forall v^h \in V_{h,0}.
	\end{equation}
The assumption that  $\alpha$ is a constant coefficient matrix and $\hat{u}$ is a polynomial of degree $k$ means $\hat{f}$ is also  a polynomial of degree $k-2$.  Then, the definition of $L^2$-projection in \eqref{equ:l2proj} implies that 
\begin{equation}\label{equ:fff111}
	(\hat{f}, \Pi_{k-1}^{0} v^{h} - v^{h}) = 0, \quad \forall v\in V_{h,0}. 
\end{equation}
Using the definition of the variational formulation \eqref{equ:varform} and the VEM variational formulation \eqref{equ:vem}, it is easy to see that 
\begin{equation}
	a_h(\hat{u}^h, v^h) - a(\hat{u}, v^h) = (\hat{f}, \Pi_{k-1}^{0} v^{h} - v^{h}), \quad \forall v^h\in V_h.
\end{equation}
Then, the orthogonal relationship \eqref{equ:fff111} implies 
\begin{equation}\label{equ:faaaa}
	a_h(\hat{u}^h, v^h) = a(\hat{u}, v^h), \quad v\in V_{h,0}.
\end{equation}
Note that the fact that $\hat{u}$ is a polynomial with degree $k$ and 
$\hat{u}^h(\mathbf{x}) = {\hat{u}({\color{red}\mathbf{x})},\forall \mathbf{x}\in \partial\Omega}$.
Using \eqref{equ:polynomial}, \eqref{equ:faaaa}, we have  \eqref{equ:projequal}.
\end{proof}

\subsection{Approximation property of the partition of the exact solution}
For the sake of simplifying the notation, we use the shorthand $B_{r_l} := B(\mathbf{x}_0, r_l)$. 
Let $u_l^h$ be the VEM solution of the problem \eqref{equ:subprob}  defined in \eqref{equ:vem_sol_ul}.  
Assume that $g_{1}(\mathbf{x})=u_{l}(\mathbf{x}),~\forall \mathbf{x}\in \partial\Omega$.
To estimate $u_{l}-u_{l}^{h}$, we first estimate $\|g_{l}\|_{H^{k+\frac{1}{2}}(\partial\Omega)}$.

\begin{lemma} \label{lem:partapppre}
Assume that $u\in W^{k+1,\infty}(\Omega)$ and  $1\leq l\leq l_{0}$. Then, we have 
    \begin{equation}\label{equ:partapph1pre}
\|g_{l}\|_{H^{k+\frac{1}{2}}(\partial\Omega)}
\lesssim r_{l}\|u\|_{W^{k+1,\infty}(\Omega)}.
    \end{equation}
   \end{lemma} 
\begin{proof} Assume that $\hat{B}_{E_{0},r}$ is defined as in  \eqref{hatrK0}.
Let $\Omega_{l}=\hat{B}_{E_{0},r_{l+1}}\cap \Omega$. Note that $\phi_{l}(\mathbf{x})=0,~ \forall   \mathbf{x}\in \mathbb{R}^{2}\backslash\hat{B}_{E_{0},r_{l+1}}$.
We observes that 
\begin{align}
g_{l}(\mathbf{x})=0,\qquad \forall   \mathbf{x}\in \partial\Omega\backslash\partial\Omega_{l}.\label{111}
\end{align}
 Let
\begin{align}
\overline{\phi}_{l}=\hat{\phi}_{l}|_{\partial\Omega}, \quad \overline{u}=\hat{u}|_{\partial\Omega},\quad 
\overline{u}_{1}=u|_{\partial\Omega}-\hat{u}|_{\partial\Omega}.\label{222}
\end{align}
Note that $u(\mathbf{x})=0$ for all $\mathbf{x}\in \partial\Omega_{l}$ and $\overline{u}(\mathbf{x})$
 is a polynomial with degree $k$. By \eqref{newnew00}, \eqref{111} and \eqref{222}, we have  
 \begin{equation}\label{0001}
 	\begin{split}
&\|g_{l}\|_{H^{k+\frac{1}{2}}(\partial\Omega)}=
\|g_{l}\|_{H^{k+\frac{1}{2}}(\partial\Omega_{l})}\lesssim \sum\limits_{m=0}^{l}
\|\overline{\phi}_{l}\|_{W^{m,\infty}(\partial\Omega_{l})}
\|\overline{u}\|_{H^{k+\frac{1}{2}-m}(\partial\Omega_{l})}\\
\lesssim &\sum\limits_{m=0}^{l}r_{l}^{-m}\|\overline{u}\|_{H^{k+\frac{1}{2}-m}(\partial\Omega_{l})}
\lesssim \sum\limits_{m=0}^{l}r_{l}^{-m}r_{l}^{m-k-\frac{1}{2}}
\|\overline{u}\|_{L^{2}(\partial\Omega_{l})}\\
\lesssim &r_{l}^{-k-\frac{1}{2}}\|\overline{u}\|_{L^{2}(\partial\Omega_{l})}
\lesssim r_{l}^{-k-\frac{1}{2}}\|\overline{u}_{1}\|_{L^{2}(\partial\Omega_{l})}.
 	\end{split}
 \end{equation}
To estimate $\|\overline{u}_{1}\|_{L^{2}(\partial\Omega_{l})}$, we introduce  $\Phi_{l}$ by 
\begin{align}
\Phi_{l}=\{\frac{\mathbf{x}}{r_{l}}|~\mathbf{x}\in \Omega_{l}\}.\label{000}
\end{align}
Furthermore, let
\begin{align}
\psi_{0}(\frac{\mathbf{x}}{r_{l}})=u(\mathbf{x})-\hat{u}(\mathbf{x}),\qquad \forall \mathbf{x}\in \Omega_{l},\label{001}
\end{align}
and 
\begin{align}
\overline{\psi}_{0}=\psi_{0}|_{\partial\Phi_{l}}.\label{002}
\end{align}
By \eqref{000}, \eqref{001} and \eqref{002}, we have
\begin{align}
\|\overline{u}_{1}\|_{L^{2}(\partial\Omega_{l})}
=r_{l}^{\frac{1}{2}}\|\overline{\psi}_{0}\|_{L^{2}(\partial\Phi_{l})}
\lesssim r_{l}^{\frac{1}{2}}\|\psi_{0}\|_{W^{k+1,\infty}(\Phi_{l})}.\label{003}
\end{align}
Note that \eqref{equ:poly}  implies
\begin{equation}\label{004}
\begin{split}
	&\|\nabla^{m}\psi_{0}\|_{L^{\infty}(\Phi_{l})}=r_{l}^{m}
\|\nabla^{m}(u-\hat{u})\|_{L^{\infty}(\Omega_{l})}\\
\lesssim & r_{l}^{m}r_{l}^{k+1-m}\|u\|_{W^{k+1,\infty}(\Omega_{l})}
\lesssim r_{l}^{k+1}\|u\|_{W^{k+1,\infty}(\Omega)}.
\end{split}
\end{equation}
Combining \eqref{0001}, \eqref{003} and \eqref{004}  gives 
\begin{equation}
\begin{split}
	&\|g_{l}\|_{H^{k+\frac{1}{2}}(\partial\Omega_{l})}\lesssim  r_{l}^{-k-\frac{1}{2}}  r_{l}^{\frac{1}{2}}\|\psi_{0}\|_{W^{k+1,\infty}(\Phi_{l})}\\
	\lesssim & r_{l}^{-k-\frac{1}{2}}  r_{l}^{\frac{1}{2}} r_{l}^{k+1}\|u\|_{W^{k+1,\infty}(\Omega)}
  \lesssim r_{l}\|u\|_{W^{k+1,\infty}(\Omega)}.
\end{split}
  \end{equation}
\end{proof}
Now,  we are in a perfect position to present the estimate for  $u_{l}-u_{l}^{h}$.

\begin{lemma} \label{lem:partapp}
Assume that $u\in W^{k+1,\infty}(\Omega)$ and  $0\leq l\leq l_{0}$. Then, we have 
    \begin{equation}\label{equ:partapph1}
\|u_{l}-u_{l}^{h}\|_{H^{1}(\Omega)}\lesssim h^{k}r_{l}\|u\|_{W^{k+1,\infty}(\Omega)},
    \end{equation}
and 
 \begin{equation}\label{equ:partapph100}
  \|u_{l}-u_{l}^{h}\|_{L^{2}(\Omega)}\lesssim h^{k+1}r_{l}\|u\|_{W^{k+1,\infty}(\Omega)}.
    \end{equation}
\end{lemma}
\begin{proof} 
Let  $\hat{B}_{E_{0},r_{-2}}=\hat{B}_{E_{0},r_{-1}}=\emptyset$. 
Note that $u_l(\mathbf{x})=0,~\forall \mathbf{x}\in \Omega\smallsetminus \hat{B}_{E_{0},r_{l+1}}$.
Combining   \cref{thm:inhomo},  \eqref{newnew00},\eqref{equ:poly} and Lemma  \ref{lem:partapppre},  we have 
\begin{align*}
&\|u_{l}-u_{l}^{h}\|_{H^{1}(\Omega)}
\lesssim h^{k}(\|u_l\|_{H^{k+1}(\Omega)}+\|g_{l}\|_{H^{k+\frac{1}{2}}(\partial\Omega)})\\
=&h^{k}(\|u_l\|_{H^{k+1}(\hat{B}_{E_{0},r_{l+1}}\cap\Omega)}+r_{l}\|u\|_{W^{k+1,\infty}(\Omega)})
\\=&h^{k}(\sum\limits_{m=0}^{k+1}
\|\overline{\phi}_{l}\|_{W^{m,\infty}(\hat{B}_{E_{0},r_{l+1}}\cap \Omega)}
\|u-\hat{u}\|_{H^{k+1-m}(\hat{B}_{E_{0},r_{l+1}}\cap\Omega)}+r_{l}\|u\|_{W^{k+1,\infty}(\Omega)})\\
 \lesssim & h^{k}(\sum\limits_{m=0}^{k+1}r_{l}^{-m}
\|u-\hat{u}\|_{H^{k+1-m}(\hat{B}_{E_{0},r_{l+1}}\cap \Omega)}+r_{l}\|u\|_{W^{k+1,\infty}(\Omega)})
\\
\lesssim &  h^{k}(\sum\limits_{m=0}^{k+1}r_{l}^{-m}r_{l}
\|u-\hat{u}\|_{W^{k+1-m,\infty}(\hat{B}_{E_{0},r_{l+1}}\cap \Omega)}+r_{l}\|u\|_{W^{k+1,\infty}(\Omega)})\\
\lesssim & h^{k}(\sum\limits_{m=0}^{k+1}r_{l}^{-m}r_{l}
r_{l}^{m}+r_{l})\|u\|_{W^{k+1,\infty}(\Omega)}\\
\lesssim & h^{k}r_{l}\|u\|_{W^{k+1,\infty}(\Omega)}.
\end{align*}
Using a similar argument, we can prove \cref{equ:partapph100}. 
\end{proof}

\subsection{A technical lemma} To prepare for the proof  of  the main result, we firstly establish a technical lemma on the local error estimates for the partition of the exaction solution. 

\begin{lemma}\label{lem:h1ball}
Assume that $u\in W^{k+1, \infty}(\Omega)$. Under the assumption that $3\leq l\leq l_{0}$, there holds
\begin{equation}\label{equ:h1ball}
\|u_{l}-u_{l}^{h}\|_{H^{1}(B_{E_0,r_{l-2}})}\lesssim h^{k+1}|\ln h|\|u\|_{W^{k+1, \infty}(\Omega)}.
 \end{equation}
\end{lemma}
\begin{proof}
Firstly, we notice that  Lemma \ref{lem:partapp} implies that there exists a constant $\gamma_{0}$ such that
\begin{align*}
 \|u_{l}-u_{l}^{h}\|_{H^{1}(B_{E_0, r_{l-1}})}\leq \gamma_{0}h^{k}r_l\|u\|_{W^{k+1, \infty}(\Omega)}.
 \end{align*}
 Without loss of generality, we assume that
 \begin{align}
 \|u_{l}-u_{l}^{h}\|_{H^{1}(B_{E_0,r_{l-2}})}\geq \gamma_{0}h^{k+1}\|u\|_{W^{k+1, \infty}(\Omega)},\label{efffs}
 \end{align}
 since otherwise we already have \eqref{equ:h1ball}. 
Combining the above two estimates gives
\begin{align}
\frac{\|u_{l}-u_{l}^{h}\|_{H^{1}(B_{E_0,r_{l-2}})}}{\|u_{l}-u_{l}^{h}\|_{H^{1}(B_{E_0,r_{l-1}})}}
\geq \frac{h}{r_{l}}=2^{-l}l^{-1}.\label{equ:efffs1}
\end{align}
Let $D_{0}=B_{E_0,r_{l-2}}$ and  $D_{m}=
\bigcup\limits_{\rho(E_0,D_{m-1})\leq h,E\in \mathcal{T}_{h}}E$  for $m\geq 1$ where $\rho(E_0,D_{m-1})$ is the distance between $E_0$ and $D_{m-1}$. Also,  let $D_{m_{0}}= B_{E_0, r_{l-1}}$.
We can  observe that there exists a constant $\gamma_1$ independent of $l$ such that 
\begin{equation}\label{equ:000000}
m_{0} = \gamma_1  2^{l-1}(l-1).
\end{equation}
Let $s_{m}=0$ if $m=1$, and $s_m$ be the minimal   positive integer satisfying $s_m \ge (m-1)\gamma_1 2^{l-1}$ if $m\ge 2$. 
Assume that $1\leq \hat{m}\leq l-1$ is the index such that
\begin{align}
\frac{\|u_{l}^{h}\|_{H^{1}(D_{s_{\hat{m}}})}}{\|u_{l}^{h}\|_{H^{1}(D_{s_{\hat{m}+1}})}}
=\max_{1\leq m\leq l-1}\frac{\|u_{l}^{h}\|_{H^{1}(D_{s_{m}})}}{\|u_{l}^{h}\|_{H^{1}(D_{s_{m+1}})}}.\label{equ:njia1}
\end{align}
According to the definition of the partition in \eqref{equ:partdef},  we can conclude that  
$u_{l}(\mathbf{x})=0$ for all $\mathbf{x}\in B_{E_0,r_{l-1}}$. It means that 
\begin{align}
u_{l}(\mathbf{x})=0,\qquad \forall \mathbf{x}\in D_{s_{\hat{m}+1}}.\label{equ:njia101}
\end{align}
By noticing    the fact that $D_{0}=B_{E_0,r_{l-2}}$ and  $D_{m_{0}}= B_{E_0,r_{l-1}}$,  we can deduce the following inequality using a simple calculation 
\begin{equation}\label{equ:ouuu00}
\frac{\|u_{l}^{h}\|_{H^{1}(D_{s_{\hat{m}}})}}{\|u_{l}^{h}\|_{H^{1}(D_{s_{\hat{m}+1}})}}
\geq \left(
\frac{\|u_{l}^{h}\|_{H^{1}(B_{E_0,r_{l-2}})}}{\|u_{l}^{h}\|_{H^{1}(B_{E_0,r_{l-1}})}}\right)^{\frac{1}{l-1}}
\geq (\frac{h}{r_{l}})^{\frac{1}{l-1}}
=\left(2^{-l}\right)^{\frac{1}{l-1}}=2^{\frac{-l}{l-1}}\geq c {\color{red},}
\end{equation}
where we have used  \eqref{equ:efffs1}, \eqref{equ:njia1}  and \eqref{equ:njia101}.

Let $\overline{s}\in [s_{\hat{m}}+1,s_{\hat{m}+1}]$ be the index such that
\begin{equation}\label{equ:00001}
\begin{split}
&\|u_{l}^{h}\|_{H^{1}(D_{\overline{s}}\backslash D_{\overline{s}-1})}
\|u_{l}^{h}\|_{L^{2}(D_{\overline{s}}\backslash D_{\overline{s}-1})}\\
 =&\min\limits_{s_{\hat{m}+1}\leq s^{'}\leq s_{\hat{m}+1}}\|u_{l}^{h}\|_{H^{1}(D_{s^{'}}\backslash D_{s^{'}-1})}
\|u_{l}^{h}\|_{L^{2}(D_{s^{'}}\backslash D_{s^{'}-1})}.
\end{split}
\end{equation}
Then, we can show that 
\begin{equation}\label{equ:new01}
\begin{split}
 &\sum\limits_{s^{'}=s_{\hat{m}}+1}^{s_{\hat{m}+1}}\|u_{l}^{h}\|_{H^{1}(D_{s^{'}}\backslash D_{s^{'}-1})}
\|u_{l}^{h}\|_{L^{2}(D_{s^{'}}\backslash D_{s^{'}-1})}\\
 \leq &
 \left(\sum\limits_{s^{'}=s_{\hat{m}}+1}^{s_{\hat{m}+1}}\|u_{l}^{h}\|_{H^{1}(D_{s^{'}}\backslash D_{s^{'}-1})}^{2}\right)^{\frac{1}{2}}
 \left(\sum\limits_{s^{'}=s_{\hat{m}}+1}^{s_{\hat{m}+1}}\|u_{l}^{h}\|_{L^{2}(D_{s^{'}}\backslash D_{s^{'}-1})}^{2}\right)^{\frac{1}{2}}.
\end{split}
\end{equation}
 Also, using \eqref{equ:000000}, we can establish the following relationship 
 \begin{align}\label{equ:00002}
 s_{\hat{m}+1}-s_{\hat{m}}=m_{1}\geq \frac{m_{0}}{l}\geq c2^{l}.
 \end{align}
By \eqref{equ:new01},\eqref{equ:00001}  and  \eqref{equ:00002}, we have
\begin{equation}\label{equ:ouuu}
\begin{split}
 &\|u_{l}^{h}\|_{H^{1}(D_{\overline{s}}\backslash D_{\overline{s}-1})}
\|u_{l}^{h}\|_{L^{2}(D_{\overline{s}}\backslash D_{\overline{s}-1})}\\
\lesssim &
\frac{1}{s_{\hat{m}+1}-s_{\hat{m}}}\sum\limits_{s^{'}=s_{\hat{m}}+1}^{s_{\hat{m}+1}}\|u_{l}^{h}\|_{H^{1}(D_{s^{'}}\backslash D_{s^{'}-1})}
\|u_{l}^{h}\|_{L^{2}(D_{s^{'}}\backslash D_{s^{'}-1})}\\
\lesssim &
 \frac{1}{2^{l}}\left(\sum\limits_{s^{'}=s_{\hat{m}}+1}^{s_{\hat{m}+1}}\|u_{l}^{h}\|_{H^{1}(D_{s^{'}}\backslash D_{s^{'}-1})}^{2}\right)^{\frac{1}{2}}
 \left(\sum\limits_{s^{'}=s_{\hat{m}}+1}^{s_{\hat{m}+1}}\|u_{l}^{h}\|_{L^{2}(D_{s^{'}}\backslash D_{s^{'}-1})}^{2}\right)^{\frac{1}{2}}\\
 \lesssim&
 \frac{1}{2^{l}}\|u_{l}^{h}\|_{H^{1}(D_{s_{\hat{m}+1}})}
 \|u_{l}^{h}\|_{L^{2}(D_{s_{\hat{m}+1}})}.
\end{split}
\end{equation}

 Let $\psi(\mathbf{x})$ be a cutoff function such that  $\psi(\mathbf{x})=1$  if $\mathbf{x}\in D_{\overline{s}}$ satisfies
$\rho(\mathbf{x},\partial D_{\overline{s}})\leq c_{1}h$ for some constant $0<c_{1}<1$,
and $\psi(\mathbf{x})=0$  if $\mathbf{x}\in  D_{\overline{s}-1}$. It is easy to deduce that   $\|\psi\|_{W^{p,\infty}(\Omega)}\leq ch^{-p}$ for any positive integer $p$.
We decompose $u_{l}^{h}(\mathbf{x})$ as 
\begin{equation}\label{equ:soldecomp}
u_{l}^{h}(\mathbf{x})=\overline{u}_{l}(\mathbf{x})+\hat{u}_{l}(\mathbf{x}),
\end{equation}
where
\begin{equation}\label{equ:soldecomp_part1}
\overline{u}_{l}(\mathbf{x})=I_{h}(\psi u_{l}^{h})(\mathbf{x}),
\end{equation}
with $I_h$ being the interpolation operator of $V_h$ defined in \eqref{equ:interpolation}. 
Then, we have
\begin{equation}\label{equ:ss00}
a_{h}^{D_{\overline{s}}}(u_{l}^{h},u_{l}^{h})= a_{h}^{D_{\overline{s}}}(u_{l}^{h},\overline{u}_{l})+
a_{h}^{D_{\overline{s}}}(u_{l}^{h},\hat{u}_{l}).
\end{equation}

We first estimate $a_{h}^{D_{\overline{s}}}(u_{l}^{h},\overline{u}_{l})$. 
Using  the property of $\psi$, we can show that 
\begin{equation*}
\begin{split}
&a_{h}^{D_{\overline{s}}}(u_{l}^{h},\overline{u}_{l})=a_{h}^{D_{\overline{s}}\backslash D_{\overline{s}-1}}(u_{l}^{h},\overline{u}_{l})\\
\leq& a_{h}^{D_{\overline{s}}\backslash D_{\overline{s}-1}}
(u_{l}^{h},u_{l}^{h})+ch^{-1}\|u_{l}^{h}\|_{H^{1}(D_{\overline{s}}\backslash D_{\overline{s}-1})}
\|u_{l}^{h}\|_{L^{2}(D_{\overline{s}}\backslash D_{\overline{s}-1})}.
\end{split}
\end{equation*}
Rearranging the above equation implies 
\begin{equation}\label{equ:ss001}
a_{h}^{D_{\overline{s}}}(u_{l}^{h},\overline{u}_{l})-a_{h}^{D_{\overline{s}}\backslash D_{\overline{s}-1}}
(u_{l}^{h},u_{l}^{h})
\leq ch^{-1}\|u_{l}^{h}\|_{H^{1}(D_{\overline{s}}\backslash D_{\overline{s}-1})}
\|u_{l}^{h}\|_{L^{2}(D_{\overline{s}}\backslash D_{\overline{s}-1})}.
\end{equation}

Moving on now to estimate $a_{h}^{D_{\overline{s}}}(u_{l}^{h},\hat{u}_{l})$.
Notice that  $\hat{u}_{l}(\mathbf{x})=0$ for all $\mathbf{x}\in \partial D_{\overline{s}}$. We can introduce a function   $\omega(x)\in V_h$ by
$\omega(\mathbf{x})=\hat{u}_{l}(x)$ if $\mathbf{x}\in D_{\overline{s}}$, and
$\omega(\mathbf{x})=0$ if $\mathbf{x}\in \Omega\backslash D_{\overline{s}}$. It follows that 
\begin{equation}\label{equ:somezeor}
	\Pi^0_k \omega(\mathbf{x}) = 0, \quad \forall \mathbf{x} \in \Omega\backslash D_{\overline{s}}.
\end{equation}
Notice that $D_{\overline{s}}\subset B_{E_{0},r_{l-1}}\subset \hat{B}_{E_{0},{r-1}}$, 
\begin{equation}\label{equ:etts}
u_{l}(\mathbf{x})=0,\quad \forall \mathbf{x}\in  D_{\overline{s}},
\end{equation}
which implies 
\begin{equation}\label{equ:etts1}
f_{l}(\mathbf{x})=\L u_{l}(\mathbf{x})=0,\quad \forall \mathbf{x}\in D_{\overline{s}}. 
\end{equation}
Then, equation \eqref{equ:somezeor} and  \eqref{equ:etts1}  tell us that 
\begin{equation}\label{equ:opp1}
\begin{split}
&a_{h}^{D_{\overline{s}}}(u_{l}^{h},\hat{u}_{l})\\
=&a_{h}(u_{l}^{h},\omega)
=(f_{l}^{h},\Pi^0_{k}\omega)
=(f_{l}^{h},\Pi^0_{k}\omega)_{D_{\overline{s}}}+
(f_{l}^{h},\Pi^0_{k}\omega)_{\Omega\backslash D_{\overline{s}}}\\
=&(0,\omega)_{D_{\overline{s}}}+(f_{l}^{h},0)_{\Omega\backslash D_{\overline{s}}}
=0+0=0.
\end{split}
\end{equation}

Combining the results  \eqref{equ:partapph1}, \eqref{equ:ouuu},\eqref{equ:ss001} and  \eqref{equ:opp1}, we have
\begin{equation}\label{equ:opp0}
\begin{split}
&a_{h}^{D_{\overline{s}-1}}(u_{l}^{h},u_{l}^{h})\\
=&a_{h}^{D_{\overline{s}}}(u_{l}^{h},u_{l}^{h})
-a_{h}^{D_{\overline{s}}\backslash D_{\overline{s}-1}}(u_{l}^{h},u_{l}^{h})\\
=&a_{h}^{D_{\overline{s}}}(u_{l}^{h},\overline{u}_{l})+a_{h}^{D_{\overline{s}}}(u_{l}^{h},\hat{u}_{l})
-a_{h}^{D_{\overline{s}}\backslash D_{\overline{s}-1}}(u_{l}^{h},u_{l}^{h})\\
=&a_{h}^{D_{\overline{s}}}(u_{l}^{h},\overline{u}_{l})
-a_{h}^{D_{\overline{s}}\backslash D_{\overline{s}-1}}(u_{l}^{h},u_{l}^{h})\\
\lesssim &
h^{-1}\|u_{l}^{h}\|_{H^{1}(D_{\overline{s}}\backslash D_{\overline{s}-1})}
\|u_{l}^{h}\|_{L^{2}(D_{\overline{s}}\backslash D_{\overline{s}-1})}\\
\lesssim & h^{-1}\frac{l}{2^{l}}\|u_{l}^{h}\|_{H^{1}(D_{s_{\hat{m}+1}})}
\|u_{l}^{h}\|_{L^{2}(D_{s_{\hat{m}+1}})}\\
\lesssim  &h^{-1}hr_{l}^{-1}l\|u_{l}^{h}\|_{H^{1}(D_{s_{\hat{m}+1}})}
\|u_{l}-u_{l}^{h}\|_{L^{2}(D_{s_{\hat{m}+1}})}\\
\lesssim &r_{l}^{-1}l\|u_{l}^{h}\|_{H^{1}(D_{s_{\hat{m}+1}})}
ch^{k+1}r_l\|u\|_{W^{k+1, \infty}(\Omega)}\\
\lesssim  &h^{k+1}l\|u_{l}^{h}\|_{H^{1}(D_{s_{\hat{m}+1}})}
\|u\|_{W^{k+1, \infty}(\Omega)}.
\end{split}
\end{equation}

To show the local error estimate \eqref{equ:h1ball},  we use  \eqref{equ:opp0} and \eqref{equ:ouuu00}, which  gives us that 
\begin{equation}\label{equ:uuuu}
\begin{split}
\|u_{l}^{h}\|_{H^{1}(D_{s_{\hat{m}}})}^{2}&\lesssim
a_{h}^{D_{s_{\hat{m}}}}(u_{l}^{h},u_{l}^{h})
\lesssim a_{h}^{D_{\overline{s}-1}}(u_{l}^{h},u_{l}^{h})
\\&
\lesssim h^{k+1}l\|u_{l}^{h}\|_{H^{1}(D_{s_{\hat{m}+1}})}\|u\|_{W^{k+1, \infty}(\Omega)}
\\&\lesssim h^{k+1}l\|u_{l}^{h}\|_{H^{1}(D_{s_{\hat{m}}})}\|u\|_{W^{k+1, \infty}(\Omega)}.
\end{split}
\end{equation}
By the definition of the partition of the exact solution in \eqref{equ:partdef}, we have
\begin{equation}\label{equ:uuuuu}
u_{l}(\mathbf{x})=0,\qquad \forall \mathbf{x}\in B_{E_0,r_{l-1}}.
\end{equation}
From \eqref{equ:uuuu} and  \eqref{equ:uuuuu},   it follows that
\begin{equation}\label{equ:newjiaadd1}
\begin{split}
	\|u_{l}-u_{l}^{h}\|_{H^{1}(B_{E_0,r_{l-2}})}
=&\|u_{l}^{h}\|_{H^{1}(B_{E_0,r_{l-2}})}
\leq \|u_{l}^{h}\|_{H^{1}(D_{s_{\hat{m}}})}\\
\lesssim &
h^{k+1}l\|u_{l}\|_{W^{k+1,\infty}(\Omega)}\lesssim h^{k+1}l_0\|u\|_{W^{k+1,\infty}(\Omega)}\\
\lesssim & h^{k+1}|\ln h|\|u\|_{W^{k+1,\infty}(\Omega)},
\end{split}
\end{equation}
which completes the proof. 
\end{proof}

\subsection{Local error estimates}  Using the results in previous subsections, we shall establish a local $H^1$ error estimate for virtual element method. 

\begin{lemma} \label{lem:lemmaadd1}
Let $1\leq \overline{l}\leq l_{0}$.
 Assume  that $E_0$ is  a given element in $\mathcal{T}_h$ and $B_{E_0,r}$  is defined as in \eqref{equ:eleset}.
 Under the same assumptions of Theorem \ref{thm:maxerror}, there holds
  \begin{align}\label{equ:lemmaadd1result2}
 \|u-u^{h}\|_{H^{1}(B_{E_0,r_{\overline{l}}})}\lesssim h^{k}r_{\overline{l}}|\ln h|^{2}\|u\|_{W^{k+1,\infty}(\Omega)}.
  \end{align}
\end{lemma}
\begin{proof}
Assume that $\hat{u}(\mathbf{x}),u_{0}(\mathbf{x})$ and $u_{l}^{h}(\mathbf{x})~(1\leq l\leq l_{0})$ are defined as in subsection \ref{ssec:cut}. 
By \eqref{equ:partition} and Lemma \ref{lem:equal}, we have
\begin{align}\label{equ:jian2}
(u-u^{h})(\mathbf{x})=\sum\limits_{l=0}^{\overline{l}+1}(u_{l}-u_{l}^{h})(\mathbf{x})+\sum\limits_{l=\overline{l}+2}^{l_{0}}(u_{l}-u_{l}^{h})(\mathbf{x}).
\end{align}
We first estimate $\left\|\sum\limits_{l=0}^{\overline{l}+1}(u_{l}-u_{l}^{h})\right\|_{H^{1}(B_{E_{0}, r_{\overline{l}}})}$.
By  \eqref{rldefine} and Lemma \ref{lem:partapp}, we have, if $0\leq l\leq \overline{l}+1$, then
\begin{align*}
\|u_{l}-u_{l}^{h}\|_{H^{1}(B_{E_{0},r_{\overline{l}}})}\leq \|u_{l}-u_{l}^{h}\|_{H^{1}(\Omega)}
\lesssim h^{k}r_{l}\|u\|_{W^{k+1,\infty}(\Omega)}
\lesssim 2^{l-\overline{l}}h^{k}r_{\overline{l}}\|u\|_{W^{k+1,\infty}(\Omega)},
\end{align*}
which implies
\begin{align}
\left\|\sum\limits_{l=0}^{\overline{l}+1}(u_{l}-u_{l}^{h})\right\|_{H^{1}(B_{E_{0},r_{\overline{l}}})}\lesssim
\sum\limits_{l=0}^{\overline{l}+1}2^{l-\overline{l}}h^{k}r_{\overline{l}}\|u\|_{W^{k+1,\infty}(\Omega)}
\lesssim
h^{k}r_{\overline{l}}\|u\|_{W^{k+1,\infty}(\Omega)}.
\label{equ:simple11}
\end{align}
Next we estimate $\left\|\sum\limits_{l=\overline{l}+2}^{l_{0}}(u_{l}-u_{l}^{h})\right\|_{H^{1}(B_{E_{0},r_{\overline{l}}})}$.  By  Lemma \ref{lem:h1ball}, we have 
\begin{equation}\label{equ:simple22}
	\begin{split}
		&\left\|\sum\limits_{l=\overline{l}+2}^{l_{0}}(u_{l}-u_{l}^{h})\right\|_{H^{1}(B_{E_{0},r_{\overline{l}}})}
\leq \sum\limits_{l=\overline{l}+2}^{l_{0}}\|u_{l}-u_{l}^{h}\|_{H^{1}(B_{E_{0},r_{\overline{l}}})}
\\
\leq &\sum\limits_{l=\overline{l}+2}^{l_{0}}\|u_{l}-u_{l}^{h}\|_{H^{1}(B_{E_0,r_{l-2}})}
\leq \sum\limits_{l=\overline{l}+2}^{l_{0}}ch^{k+1}|\ln h|\|u\|_{W^{k+1,\infty}(\Omega)}
\\\lesssim &h^{k}r_{\overline{l}}|\ln h|^{2}\|u\|_{W^{k+1,\infty}(\Omega)}.
	\end{split}
\end{equation}
Inserting \eqref{equ:simple11} and  \eqref{equ:simple22} into \eqref{equ:jian2}, we get the desired result \eqref{equ:lemmaadd1result2}.
\end{proof}

Now, we shall estimate the local error $\|u-u^{h}\|_{H^{1}(E_{0})}$  and $\|u-u^{h}\|_{L^{2}(E_{0})}$. 

\begin{lemma} \label{lem:localestimate}
  Assume that $E_0$ is  a given element in $\mathcal{T}_h$ and 
$u\in W^{k+1,\infty}(\Omega)$. Then
\begin{align}\label{equ:h1local}
  \|u-u^{h}\|_{H^{1}(E_0)}\lesssim h^{k+1}|\ln h|^{2}\|u\|_{W^{k+1,\infty}(\Omega)}.
  \end{align}
Furthermore, under the assumption that   $k\geq 2$, there holds
    \begin{align}\label{equ:l2local}
 \|u-u^{h}\|_{L^{2}(E_0)}\lesssim h^{k+2}|\ln h|^{4}(\|u\|_{W^{k+1,\infty}(\Omega)}+\|f\|_{H^{k}(\Omega)}).
  \end{align}
\end{lemma}

\begin{proof}
	Let $\overline{l}=1$. By \eqref{equ:lemmaadd1result2}, we have
	\begin{equation}
		\begin{split}
			 \|u-u^{h}\|_{H^{1}(E_{0})} &\leq
 \|u-u^{h}\|_{H^{1}(B_{E_0,r_{\overline{l}}})}\\
& \lesssim h^{k}r_{\overline{l}}|\ln h|^{2}\|u\|_{W^{k+1,\infty}(\Omega)}\\
 &\lesssim  h^{k+1}|\ln h|^{2}\|u\|_{W^{k+1,\infty}(\Omega)},
		\end{split}
	\end{equation}
which complete the proof of \eqref{equ:h1local}. 

 To show \eqref{equ:l2local}, let $\theta_{0}(\mathbf{x})$ be the cutoff function
satisfying  $\theta_{0}(\mathbf{x})=1$ if $\mathbf{x}\in E_0$, and
$\theta_{0}(\mathbf{x})=0$ if $\mathbf{x}\in \Omega\backslash B_{E_0, r_{1}}$.
Assume that  $\theta_{1}(\mathbf{x})$ satisfies the following problem
\begin{eqnarray}
\begin{cases}
-\nabla\cdot (\alpha\nabla \theta_{1}(\mathbf{x}))=\theta_{0}(\mathbf{x})(u-u^{h})(\mathbf{x}),\quad &\text{in}\quad \Omega,\\
  \theta_{1}(\mathbf{x})=0, &\text{on}\quad \partial\Omega.\label{equation3.1}
\end{cases}
\end{eqnarray}
 Assume also that $f^{h}=\Pi_{k-1}^{0}f$.
By the definition of the cutoff function $\theta_0(x)$, it is easy to deduce that 
\begin{equation}\label{equ:Identity}
	\begin{split}
		 \|u-u^{h}\|_{L^{2}(E_0)}^{2}\leq
 &\left(u-u^{h},\theta_{0}(u-u^{h})\right) \\
 =&(u-u^{h},-\nabla.(\alpha\nabla \theta_{1}))\\
=&a(u-u^{h},\theta_{1}-I_{h}\theta_{1})+a(u-u^{h},I_{h}\theta_{1})\\
=&a(u-u^{h},\theta_{1}-I_{h}\theta_{1})+(f,I_{h}\theta_{1})-a(u^{h},I_{h}\theta_{1})
\\=&a(u-u^{h},\theta_{1}-I_{h}\theta_{1})+(f-f^{h},I_{h}\theta_{1})
+[a_{h}(u^{h},I_{h}\theta_{1})-a(u^{h},I_{h}\theta_{1})]
\\:=&I_{1}+I_{2}+I_{3}.
	\end{split}
\end{equation}
We first concentrate on the estimate of $I_1$.  For such purpose, we make an observation that  we can split $I_1$ as
 \begin{equation}
 	\begin{split}
 		I_{1}&=a_{B_{E_0, r_{2}}}(u-u^{h},\theta_{1}-I_{h}\theta_{1})
+\sum\limits_{l=2}^{l_{0}}
a_{B_{E_0,r_{l+1}}\backslash B_{E_0, r_{l}}}(u-u^{h},\theta_{1}-I_{h}\theta_{1})
\\&:=\sum\limits_{l=1}^{l_{0}}I_{1,l}.
 	\end{split}
 \end{equation}
It is sufficient to estimate $I_{1,l}$ for $1\le l \le l_0$. 

Lemma \ref{lem:lemmaadd1} implies that 
\begin{equation}\label{equ:I1est}
	\begin{split}
		|I_{1,1}|&\lesssim \|u^{h}-u\|_{H^{1}(B_{E_0,r_{2}})}
\|\theta_{1}-I_{h}\theta_{1}\|_{H^{1}(B_{E_0,r_{2}})}
\\&
\lesssim h^{k}r_{2}|\ln h|^{2}\|u\|_{W^{k+1,\infty}(\Omega)}
h\|\theta_{1}\|_{H^{2}(\Omega)}
\\&
\lesssim h^{k+1}|\ln h|^{2}\|u\|_{W^{k+1,\infty}(\Omega)}
h\|\theta_{0}(u-u^{h})\|_{L^{2}(\Omega)}
\\&
\lesssim h^{k+2}|\ln h|^{2}\|u\|_{W^{k+1,\infty}(\Omega)}\|\theta_{0}(u-u^{h})\|_{L^{2}(\Omega)}.
	\end{split}
\end{equation}

To estimate  $I_{1,l}$ for $l\geq 2$, notice that 
$\theta_{0}(\mathbf{x})=0$ if $\mathbf{x}\in \Omega\backslash B_{E_0, r_{1}}$.
Using the  Theorem \ref{thm:reg}, we have the following regularity result for $\theta_1$
\begin{equation}\label{equ:theta1estimate}
	\|\theta_{1}\|_{H^{3}(\Omega\backslash B_{E_0, r_{l}})}\lesssim r_{l}^{-1}|\ln r_{l}|\|\theta_{0}(u-u^{h})\|_{L^{2}(\Omega)},
\end{equation}
for $l\ge 2$.  From  \eqref{equ:lemmaadd1result2} and  \eqref{equ:theta1estimate}, we can derivate that 

\begin{equation}\label{equ:Ilest}
	\begin{split}
		|I_{1,l}|&\lesssim
\|u^{h}-u\|_{H^{1}(B_{E_0,r_{l+1}}\backslash B_{E_0,r_{l}})}
\|\theta_{1}-I_{h}\theta_{1}\|_{H^{1}(B_{E_0, r_{l+1}}\backslash B_{E_0,r_{l}})}
\\&
\lesssim h^{k}r_{l+1}|\ln h|^{2}\|u\|_{W^{k+1,\infty}(\Omega)}
h^{2}\|\theta_{1}\|_{H^{3}(\Omega\backslash B_{E_0,r_{l}})}
\\&
\lesssim h^{k}r_{l}|\ln h|^{2}\|u\|_{W^{k+1,\infty}(\Omega)}
h^{2}r_{l}^{-1}|\ln r_{l}|\|\theta_{0}(u-u^{h})\|_{L^{2}(\Omega)}
\\&
\lesssim h^{k+2}|\ln h|^{3}\|u\|_{W^{k+1,\infty}(\Omega)}\|\theta_{0}(u-u^{h})\|_{L^{2}(\Omega)}.
	\end{split}
\end{equation}
Combining the estimates of \eqref{equ:I1est} and \eqref{equ:Ilest},  we get
\begin{equation}
|I_{1}|\lesssim h^{k+2}|\ln h|^{4}\|u\|_{W^{k+1,\infty}(\Omega)}\|\theta_{0}(u-u^{h})\|_{L^{2}(\Omega)}.\label{I1estimate}
\end{equation}

Now, we turn to estimate $I_2$. By the definition of the projection operator $\Pi_{k-1}^0$,   there holds
\begin{align*}
(f-f^{h},q)_{K}=
0,\qquad\forall q\in \mathbb{P}_{k-1}(K)
\end{align*}
for any $K\in
\mathcal{T}_{h}$. We can deduce that 
\begin{equation}\label{equ:I2split}
	\begin{split}
		I_{2}&=(f-f^{h},I_{h}\theta_{1})=(f-\Pi_{k-1}^{0}f,I_{h}\theta_{1}-\Pi_{k-1}^{0}\theta_{1})\\&
 =(f-\Pi_{k-1}^{0}f,I_{h}\theta_{1}-\theta_{1})
 +(f-\Pi_{k-1}^{0}f,\theta_{1}-\Pi_{k-1}^{0}\theta_{1})\\&
 :=I_{2,1}+I_{2,2}
	\end{split}
\end{equation}
We only  need to estimate $I_{2,1}$ and $I_{2,2}$.  We start with the estimation of $I_{2,1}$. 
For $I_{2,1}$, we have 
\begin{align*}
|I_{2,1}|&\lesssim \|f-\Pi_{k-1}^{0}f\|_{L^{2}(\Omega)}\|I_{h}\theta_{1}-\theta_{1}\|_{L^{2}(\Omega)}
\lesssim
h^{k}\|f\|_{H^{k}(\Omega)}h^{2}\|\theta_{1}\|_{H^{2}(\Omega)}
\\&
\lesssim h^{k+2}\|f\|_{H^{k}(\Omega)}\|\theta_{0}(u-u^{h})\|_{L^{2}(\Omega)}. 
\end{align*}
Similarly, for $I_{2,2}$, we have 
\begin{equation}
	\begin{split}
		|I_{2,2}|&\lesssim \|f-\Pi_{k-1}^{0}f\|_{L^{2}(\Omega)}\|\Pi_{k-1}^{0}\theta_{1}-\theta_{1}\|_{L^{2}(\Omega)}
\lesssim
h^{k}\|f\|_{H^{k}(\Omega)}h^{2}\|\theta_{1}\|_{H^{2}(\Omega)}
\\&
\lesssim h^{k+2}\|f\|_{H^{k}(\Omega)}\|\theta_{0}(u-u^{h})\|_{L^{2}(\Omega)}.
	\end{split}
\end{equation}
Plugging the estimates of $I_{2,1}$ and $I_{2,2}$ into \eqref{equ:I2split}, we have
\begin{align}
|I_{2}|\lesssim h^{k+2}\|f\|_{H^{k}(\Omega)}\|\theta_{0}(u-u^{h})\|_{L^{2}(\Omega)}.\label{I2estimate}
\end{align}

Then, we move to consider the estimation of $I_3$. We proceed as 
\begin{equation}\label{equ:I3_def}
	\begin{split}
		I_{3} =&a_{h}(u^{h},I_{h}\theta_1)-a(u^{h},I_{h}\theta_1)=
		\sum\limits_{E\in\mathcal{T}_h}[a_{h}^{E}(u^{h},I_{h}\theta_1)-a^{E}(u^{h},I_{h}\theta_1)]\\
=&\sum\limits_{E\in\mathcal{T}_h}[a_{h}^{E}(u^{h}-\Pi_{k}^{0}u,I_{h}\theta_1)-a^{K}(u^{h}-\Pi_{k}^{0}u,I_{h}\theta_1)]
\\
=&\sum\limits_{E\in\mathcal{T}_h}[a_{h}^{E}(u^{h}-\Pi_{k}^{0}u,I_{h}\theta_1-\Pi_{2}^{0}\theta_1)
-a^{E}(u^{h}-\Pi_{k}^{0}u,I_{h}\theta_1-\Pi_{2}^{0}\theta_1)]
\\
=&\sum\limits_{E\subset B_{E_0,r_{2}}}[a_{h}^{E}(u^{h}-\Pi_{k}^{0}u,I_{h}\theta_1-\Pi_{2}^{0}\theta_1)-a^{E}(u^{h}-\Pi_{k}^{0}u,I_{h}\theta_1-\Pi_{2}^{0}\theta_1)]
+\\&
\sum\limits_{l=2}^{l_{0}}\sum\limits_{E\subset B_{E_0, r_{l+1}}\backslash B_{E_0, r_{l}}}[a_{h}^{E}(u^{h}-\Pi_{k}^{0}u,I_{h}\theta_1-\Pi_{2}^{0}\theta_1)-a^{E}(u^{h}-\Pi_{k}^{0}u,I_{h}\theta_1-\Pi_{2}^{0}\theta_1)]
\\:=&\sum\limits_{l=1}^{l_{0}}I_{3,l}.
	\end{split}
\end{equation}
 We first estimate $I_{3,1}$. Lemma \ref{lem:lemmaadd1} implies
\begin{equation}
	\begin{split}
		\|u^{h}-\Pi_{k}^{0}u\|_{H^{1}(B_{E_0,r_{2}})}
&\leq \|u^{h}-u\|_{H^{1}(B_{E_0,r_{2}})}+\|u-\Pi_{k}^{0}u\|_{H^{1}(B_{E_0,r_{2}})}
\\&
\lesssim h^{k}r_{2}|\ln h|^{2}\|u\|_{W^{k+1,\infty}(\Omega)}+h^{k}r_{2}\|u\|_{W^{k+1,\infty}(\Omega)}
\\&\lesssim h^{k+1}|\ln h|^{2}\|u\|_{W^{k+1,\infty}(\Omega)},
	\end{split}
\end{equation}
and
\begin{equation}
	\begin{split}
		\|I_{h}\theta_1-\Pi_{2}^{0}\theta_1\|_{H^{1}(B_{E_0,r_{2}})}
&\leq \|I_{h}\theta_1-\theta_1\|_{H^{1}(B_{E_0,r_{2}})}+\|\theta_1-\Pi_{2}^{0}\theta_1\|_{H^{1}(B_{E_0,r_{2}})}
\\&
\lesssim h\|\theta_1\|_{H^{2}(\Omega)}+h\|\theta_1\|_{H^{2}(\Omega)}
\lesssim h\|\theta_1\|_{H^{2}(\Omega)}\\&\leq h\|\theta_{0}(u-u^{h})\|_{L^{2}(\Omega)}.
	\end{split}
\end{equation}
Combining those two  estimates, we have
\begin{equation}
	\begin{split}
		&\left|\sum\limits_{E\subset B_{E_0,r_{2}}}a^{E}(u^{h}-\Pi_{k}^{0}u,I_{h}\theta_1-\Pi_{2}^{0}\theta_1)\right|\\
\lesssim &\|u^{h}-\Pi_{k}^{0}u\|_{H^{1}(B_{E_0,r_{2}})}\|I_{h}\theta_1-\Pi_{2}^{0}\theta_1\|_{H^{1}(B_{E_0,r_{2}})}
\\
\lesssim &h^{k+1}|\ln h|^{2}\|u\|_{W^{k+1,\infty}(\Omega)}h\|\theta_1\|_{H^{2}(\Omega)}
\\
\lesssim & h^{k+2}|\ln h|^{2}\|u\|_{W^{k+1,\infty}(\Omega)}\|\theta_{0}(u-u^{h})\|_{L^{2}(\Omega)}.
	\end{split}
\end{equation}
Similarly, we have
\begin{align*}
&\left|\sum\limits_{E\subset B_{E_0,r_{2}}}a_{h}^{E}(u^{h}-\Pi_{k}^{0}u,I_{h}\theta_1-\Pi_{2}^{0}\theta_1)\right|\\
\lesssim & h^{k+2}|\ln h|^{2}\|u\|_{W^{k+1,\infty}(\Omega)}\|\theta_{0}(u-u^{h})\|_{L^{2}(\Omega)}.
\end{align*}
Using those two estimates, we can deduce that 
\begin{equation}
|I_{3,1}|\lesssim h^{k+2}|\ln h|^{2}\|u\|_{W^{k+1,\infty}(\Omega)}\|\theta_{0}(u-u^{h})\|_{L^{2}(\Omega)}.\label{I31est}
\end{equation}
 We are moving to  estimate $I_{3,l}$ for $2\leq l\leq l_{0}$. Using the regularity\eqref{equ:theta1estimate}, we have
\begin{equation}
	\begin{split}
		\|I_{h}\theta_1-\Pi_{2}^{0}\theta_1\|_{H^{1}(B_{E_0,r_{l+1}}\backslash B_{E_0,r_{l}})}
&\lesssim h^{2}\|\theta_1\|_{H^{3}(B_{E_0,r_{l+1}}\backslash B_{E_0,r_{l}})}
\\&
\lesssim h^{2}r_{l}^{-1}|\ln h|\|\theta_{0}(u-u^{h})\|_{L^{2}(\Omega)},
	\end{split}
\end{equation}
and
\begin{equation}
	\begin{split}
	\|I_{h}\theta_1-\theta_1\|_{H^{1}(B_{E_0,r_{l+1}}\backslash B_{E_0,r_{l}})}
\lesssim h^{2}r_{l}^{-1}|\ln h|\|\theta_{0}(u-u^{h})\|_{L^{2}(\Omega)}.	
	\end{split}
\end{equation}
Also,  \eqref{equ:lemmaadd1result2} implies
\begin{equation}
	\begin{split}
		&\|u^{h}-\Pi_{k}^{0}u\|_{H^{1}(B_{E_0, r_{l+1}}\backslash B_{E_0, r_{l}})}
\\
\lesssim & \|u^{h}-u\|_{H^{1}(B_{E_0,r_{l+1}}\backslash B_{E_0,r_{l}})}
+\|u-\Pi_{k}^{0}u\|_{H^{1}(B_{E_0,r_{l+1}}\backslash B_{E_0,r_{l}})}
\\
\lesssim & h^{k}r_{l}|\ln h|^{2}\|u\|_{W^{k+1,\infty}(\Omega)}
+h^{k}r_{l}\|u\|_{W^{k+1,\infty}(\Omega)}
\\
\lesssim &h^{k}r_{l}|\ln h|^{2}\|u\|_{W^{k+1,\infty}(\Omega)}.
	\end{split}
\end{equation}

Using  the above three estimates and the definition of $I_3$, for  $2\leq l\leq l_{0}$, we can show 
\begin{equation}\label{equ:over1}
	\begin{split}
		&\left|\sum\limits_{E\subset B_{E_0,r_{l+1}}\backslash B_{E_0,r_{l}}}a^{E}(u^{h}-\Pi_{k}^{0}u,I_{h}\theta_1-\Pi_{2}^{0}\theta_1)\right|
\\
\lesssim &\|u^{h}-\Pi_{k}^{0}u\|_{H^{1}(B_{E_0, r_{l+1}}\backslash B_{E_0,r_{l}})}
\|I_{h}\theta_1-\Pi_{2}^{0}\theta_1\|_{H^{1}(B_{E_0, r_{l+1}}\backslash B_{E_0,r_{l}})}
\\
\lesssim &\|u^{h}-\Pi_{k}^{0}u\|_{H^{1}(B_{E_0, r_{l+1}}\backslash B_{E_0,r_{l}})}
\|\theta_1-\Pi_{2}^{0}\theta_1\|_{H^{1}(B_{E_0, r_{l+1}}\backslash B_{E_0,r_{l}})}
+ \\
&\|u^{h}-\Pi_{k}^{0}u\|_{H^{1}(B_{E_0, r_{l+1}}\backslash B_{E_0,r_{l}})}
\|I_{h}\theta_1-\theta_1\|_{H^{1}(B_{E_0, r_{l+1}}\backslash B_{E_0,r_{l}})}
\\
\lesssim & h^{k}r_{l}|\ln h|^{2}\|u\|_{W^{k+1,\infty}(\Omega)}
h^{2}r_{l}^{-1}|\ln h|\|\theta_{0}(u-u^{h})\|_{L^{2}(\Omega)} +\\
&h^{k}r_{l}|\ln h|^{2}\|u\|_{W^{k+1,\infty}(\Omega)}
h^{2}r_{l}^{-1}|\ln h|\|\theta_{0}(u-u^{h})\|_{L^{2}(\Omega)}
\\
\lesssim &h^{k+2}|\ln h|^{3}\|u\|_{W^{k+1,\infty}(\Omega)}
\|\theta_{0}(u-u^{h})\|_{L^{2}(\Omega)}.
	\end{split}
\end{equation}
Similarly, we have
\begin{equation}\label{equ:over2}
	\begin{split}
		&\left|\sum\limits_{E\subset B_{E_0,r_{l+1}}\backslash B_{E_0,r_{l}}}a_{h}^{E}(u^{h}-\Pi_{k}^{0}u,I_{h}\theta_1-\Pi_{2}^{0}\theta_1)\right|
\\&\lesssim h^{k+2}|\ln h|^{3}\|u\|_{W^{k+1,\infty}(\Omega)}
\|\theta_{0}(u-u^{h})\|_{L^{2}(\Omega)}.
	\end{split}
\end{equation}
Substituting \eqref{equ:over1} and \eqref{equ:over2} into \eqref{equ:I3_def}, we have
\begin{align}
|I_{3,l}|\lesssim  h^{k+2}|\ln h|^{3}\|u\|_{W^{k+1,\infty}(\Omega)}
\|\theta_{0}(u-u^{h})\|_{L^{2}(\Omega)}.\label{I3lestimate}
\end{align}
Combining \eqref{I31est} and \eqref{I3lestimate} gives
\begin{align}
|I_{3}|\lesssim h^{k+2}|\ln h|^{4}\|u\|_{W^{k+1,\infty}(\Omega)}\|\theta_{0}(u-u^{h})\|_{L^{2}(\Omega)}.\label{I3est}
\end{align}
By \eqref{equ:Identity}, \eqref{I1estimate}, \eqref{I2estimate} and  \eqref{I3est}, we have
\begin{align*}
 \|\theta_{0}(u-u^{h})\|_{L^{2}(\Omega)}^{2}\lesssim
 h^{k+2}|\ln h|^{4}\left(\|u\|_{W^{k+1,\infty}(\Omega)}+\|f\|_{H^{k}(\Omega)}\right)
 \|\theta_{0}(u-u^{h})\|_{L^{2}(\Omega)}.
\end{align*}
This implies
\begin{align*}
\|u-u^{h}\|_{L^{2}(E_0)}\leq
\|\theta_{0}(u-u^{h})\|_{L^{2}(\Omega)}
\lesssim h^{k+2}|\ln h|^{4}(\|u\|_{W^{k+1,\infty}(\Omega)}+\|f\|_{H^{k}(\Omega)}).
\end{align*}
This concludes the proof of  the local $L^2$ error estimate. 
\end{proof}

\subsection{An inverse estimate for VEM functions} It is well known that the classical 
polynomial inverse estimates \cite{BS2008,Ci2002} are no longer valid for VEM functions. In this subsection, we establish an inverse estimate using the maximum principle of harmonic function \cite{BS2018}.

\begin{lemma} \label{lem:inverse}
Assume that $v\in V_h$.  Under the same assumptions of Theorem \ref{thm:maxerror}, there holds
    \begin{align}\label{equ:inverse}
 \|v\|_{L^{\infty}(E_0)}\lesssim h^{-1}\|v\|_{L^{2}(E_0)}
  \end{align}
\end{lemma}
\begin{proof}
  Brenner and Sung (See \cite[Lemma 3.3]{BS2018}) showed 
  the following maximum principle
\begin{align}
 \|v\|_{L^{\infty}(E_{0})}\lesssim (\|v\|_{L^{\infty}(\partial E_{0})}+\|v\|_{H^{1}(E_{0})}),\qquad  \forall v\in H^{1}(\Omega).\label{0000ppp}
  \end{align}
  The key observation is that $v|_{\partial E_0}$ is a polynomial and the standard polynomial inverse estimates \cite{BS2008, Ci2002} are applicable, which implies 
  \begin{align*}
  \|v\|_{L^{\infty}(\partial E_0)}\lesssim h_{E_0}^{-\frac{1}{2}}\|v\|_{L^{2}(\partial E_0)}.
  \end{align*}
  Using the scaled  trace inequality for $H^1$ functions\cite{BGS2017}, we obtain
  \begin{align*}
   \|v\|_{L^{2}(\partial E_0)}\lesssim h_{E_0}^{-\frac{1}{2}}\|v\|_{L^{2}(E_0)}+h_{E_0}^{\frac{1}{2}}\|\nabla v\|_{L^{2}(E_0)}.
  \end{align*}
Combining the above two estimates, we have
  \begin{align*}
  \|v\|_{L^{\infty}(\partial E_0)}\lesssim h_{E_0}^{-1}\|v\|_{L^{2}(E_0)}+\|\nabla v\|_{L^{2}(E_0)}^.
  \end{align*}
Also,  the inverse estimate in \cite[Theorem 3.6]{ChHu2018} implies
  \begin{align}
   \|\nabla v\|_{L^{2}(E_0)}\lesssim h_{E_0}^{-1}\|v\|_{L^{2}(E_0)}.\label{0000000}
  \end{align}
   Using the above two estimates, it is relatively easy to deduce \eqref{equ:inverse}. 
\end{proof}

\subsection{Proof of the main result} With the above preparation, we are now in a perfect position to present the proof of our main numerical results.   Before we start,  we  recall the approximation property of the $L^2$-projection operator $\Pi^0_k$ as defined in \eqref{equ:l2proj}. For $\Pi^0_k$, the following approximation result holds\cite[Theorem 1.45]{DPDR2020}
\begin{equation}\label{equ:l2approxprop}
	|v-\Pi^0_kv |_{W^{m,p}(E_0)} \le h^{s-m}|v|_{W^{s,m}(E_0)},
\end{equation}
for any $v\in W^{s,p}(E_0)$,  $m\in \{0, 1, \ldots, s\}$, and  $p\in [1, \infty]$

We begin with the maximum error estimate of the difference between the gradient of  exact solution and the gradient of the  projection of the virtual element method solution. 
For such purpose, let $E_0 \in \mathcal{T}_h$ such that 
\begin{equation}
	\|\nabla u - \nabla  \Pi_k^{\nabla} u^h\|_{L^{\infty}(\Omega)} = \|\nabla u - \nabla  \Pi_k^{\nabla} u^h\|_{L^{\infty}(E_0)}.
\end{equation}
The key fact is that $\nabla  \Pi_k^{0} u - \nabla  \Pi_k^{\nabla} u^h$  and $\nabla  \Pi_k^{\nabla} u - \nabla  \Pi_k^{\nabla} u^h$ are polynomials so the inverse error estimate \cite[Corollary 1.29]{DPDR2020} for polynomials on general polygonal meshes can be applied. In addition, for the projection operator $ \Pi_k^{\nabla}$, we have the following boundedness property \cite[(2.36)]{BGS2017} 
\begin{equation}\label{equ:projbnd}
	 \|\nabla\Pi_k^{\nabla} v\|_{L^2(E_0)} \lesssim  \|\nabla v\|_{L^2(E_0)}, \quad \forall v\in H^1(E_0).
\end{equation}
 Then, \eqref{equ:h1local}, \eqref{equ:l2approxprop}, and  \eqref{equ:projbnd} imply 
\begin{equation}
	\begin{split}
		&\|\nabla u - \nabla \Pi_k^{\nabla} u^h\|_{L^{\infty}(\Omega)} \\
		=& \|\nabla u - \nabla  \Pi_k^{\nabla}u^h\|_{L^{\infty}(E_0)}\\
		\le &\|\nabla u - \nabla  \Pi^0_k u\|_{L^{\infty}(E_0)} + \|\nabla \Pi^0_ku - \nabla  \Pi_k^{\nabla} u\|_{L^{\infty}(E_0)} +  \\
		& \|\nabla  \Pi_k^{\nabla} u - \nabla  \Pi_k^{\nabla} u^h\|_{L^{\infty}(E_0)}\\
		\lesssim &\|\nabla u - \nabla  \Pi^0_k u\|_{L^{\infty}(E_0)}+
		h^{-1}\|\nabla \Pi^0_k u - \nabla  \Pi_k^{\nabla} u\|_{L^{2}(E_0)} +\\
		& h^{-1}\|\nabla  \Pi_k^{\nabla} u - \nabla  \Pi_k^{\nabla}u^h\|_{L^2(E_0)}\\
		\lesssim &\|\nabla u - \nabla  \Pi^0_k u\|_{L^{\infty}(E_0)}+
		h^{-1}\|\nabla u - \nabla  \Pi_k^{0} u\|_{L^{2}(E_0)} +\\
		& h^{-1}\|\nabla u - \nabla  \Pi_k^{\nabla} u\|_{L^{2}(E_0)}  + h^{-1}\|\nabla  \Pi_k^{\nabla} u - \nabla  \Pi_k^{\nabla}u^h\|_{L^2(E_0)}\\
		\lesssim &h^{k}{\|u\|_{W^{k+1,\infty}(\Omega)}}+ h^{k-1}{\|u\|_{H^{k+1}(E_0)}} +  h^{-1}\|\nabla  u - \nabla  u^h\|_{L^2(E_0)}\\
		\lesssim &h^{k}{\|u\|_{W^{k+1,\infty}(\Omega)}}+ h^{k}{\|u\|_{W^{k+1, \infty}(E_0)}} +  h^{-1}\|\nabla  u - \nabla  u^h\|_{L^2(E_0)}\\
		\lesssim& h^{k}{|\ln h|\|u\|_{W^{k+1,\infty}(\Omega)}}+ h^{-1}h^{k+1}|\ln h|^{2}\|u\|_{W^{k+1,\infty}(\Omega)}\\
		\lesssim & h^{k}|\ln h|^{2}\|u\|_{W^{k+1,\infty}(\Omega)},
	\end{split}
\end{equation}
where we have used the $L^{\infty}$ to $L^2$ inverse estimate for polynomial on polygons  in obtaining  the second inequality.  This completes the proof of \eqref{equ:h1max}.

Then, we are proving the $L^{\infty}$ error estimate. Assume $k\ge 2$. In that case, we can analogously let  $E_0 \in \mathcal{T}_h$ satisfy 
\begin{equation}
	\|u -  u^h\|_{L^{\infty}(\Omega)} = \| u -  u^h\|_{L^{\infty}(E_0)}.
\end{equation}
Then, Lemma \ref{lem:localestimate}, Lemma \ref{lem:inverse},  and the approximation property \eqref{equ:l2approxprop} give that 
\begin{equation}
	\begin{split}
		&\|u-u^{h}\|_{L^{\infty}(\Omega)}\\
		=&\|u-u^{h}\|_{L^{\infty}(E_{0})}\\
		\leq &\|u-\Pi_{k}^{0}u\|_{L^{\infty}(E_{0})}
+\|\Pi_{k}^{0}u-u^{h}\|_{L^{\infty}(E_{0})}
\\
\lesssim & h^{k+1}\|u\|_{W^{k+1,\infty}(\Omega)}
+h^{-1}\|\Pi_{k}^{0}u-u^{h}\|_{L^{2}(E_{0})}
\\
\lesssim &h^{k+1}\|u\|_{W^{k+1,\infty}(\Omega)}+
h^{-1}\left[\|\Pi_{k}^{0}u-u\|_{L^{2}(E_{0})}+\|u-u^{h}\|_{L^{2}(E_{0})}\right]
\\
\lesssim & h^{k+1}\|u\|_{W^{k+1,\infty}(\Omega)}
\\
&+h^{-1}\left[h^{k+2}\|u\|_{W^{k+1,\infty}(E_{0})}
+h^{k+2}|\ln h|^{4}(\|u\|_{W^{k+1,\infty}(\Omega)}+\|f\|_{H^{k}(\Omega)})\right]
\\
\lesssim & h^{k+1}|\ln h|^{4}(\|u\|_{W^{k+1,\infty}(\Omega)}+\|f\|_{H^{k}(\Omega)}).
	\end{split}
\end{equation}
It concludes the proof of \eqref{equ:l2max}.

\section{Numerical Examples}
In this section, we present a numerical example to validate our theoretical results.  In all the following numerical, the stabilizing bilinear form $S^E(\cdot, \cdot)$ is chosen as  \cite{BBCMMR2013}
\begin{equation}
	S^E(\phi_i-\Pi_k^{\nabla} \phi_i, \phi_j-\Pi_k^{\nabla} \phi_j)
	= \sum_{r=1}^{N_E} \chi_r(\phi_i-\Pi_k^{\nabla} \phi_i)\chi_r(\phi_j-\Pi_k^{\nabla} \phi_j),
\end{equation}
where $\chi_1, \cdots, \chi_{N_E}$ are the basis functions of the dual space of $W_k(E)$.

The optimal convergence of the maximal norm errors  shall be measured using discrete maximal norm error at vertices of mesh. Let $\mathcal{N}_h$ denote the set of all vertices of $\mathcal{T}_h$. We shall consider the discrete $L^{\infty}$ norm of $u-u^h$ as $\|u-u^h\|_{L^{\infty}} = \max_{\mathbf{p}\in \mathcal{N}_h}|u(\mathbf{p})-u^h(\mathbf{p})|$.  
In the virtual element method,  the piecewise  gradient of virtual element solution $u^h$ at vertices is not directly available. In practice, we use the piecewise  gradient of the projection $\Pi_k^{\nabla}u^h$ to approximate $\nabla u^h$. We take simple averaging to obtain the gradient of  $\Pi_k^{\nabla}u^h$ at a mesh vertex $\mathbf{p}\in\mathcal{N}_h$, which is denoted by $\overline{\nabla} \Pi_k^{\nabla}u^h(\mathbf{p})$.  The discrete $W^{1, \infty}$ norm of $u-u^h$ is defined as $\|u-u^h\|_{W^{1,\infty}} = \max_{p\in \mathcal{N}_h}|\nabla u(\mathbf{p})-\overline{\nabla} \Pi_k^{\nabla}u^h(\mathbf{p})|$. Let $N$ be the number of vertices of $\mathcal{T}_h$.

\begin{figure}[!h]
   \centering
  \subcaptionbox{$\mathcal{T}_{h,1}$\label{fig:hexagonal_mesh}}
   {\includegraphics[width=0.32\textwidth]{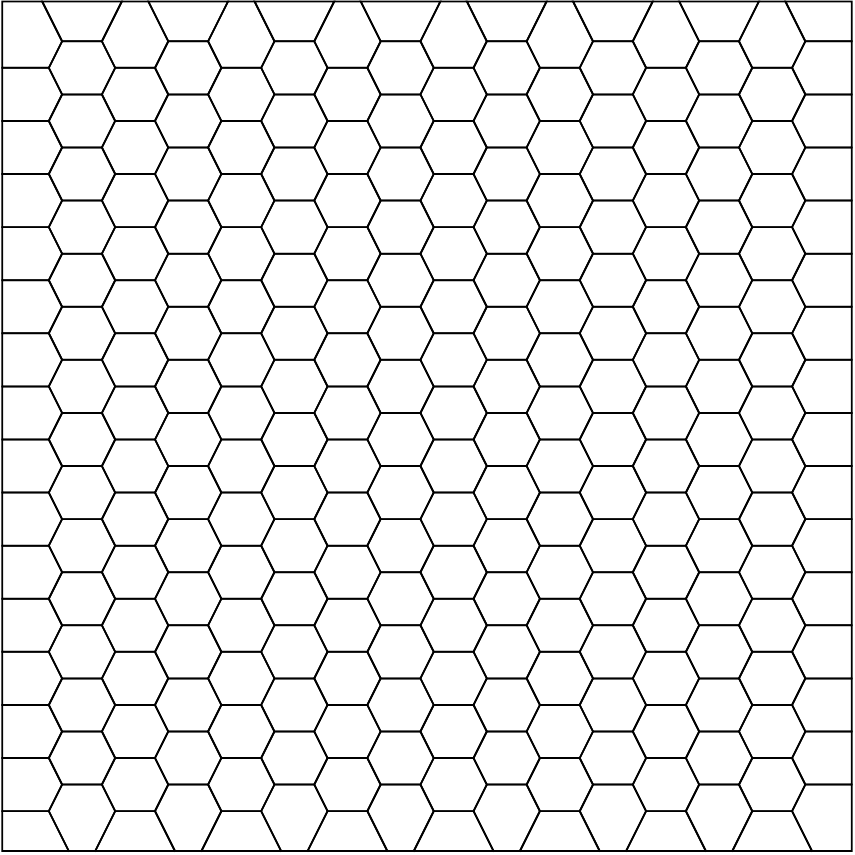}}
   \subcaptionbox{$\mathcal{T}_{h,2}$\label{fig:randquad_mesh}}
   {\includegraphics[width=0.32\textwidth]{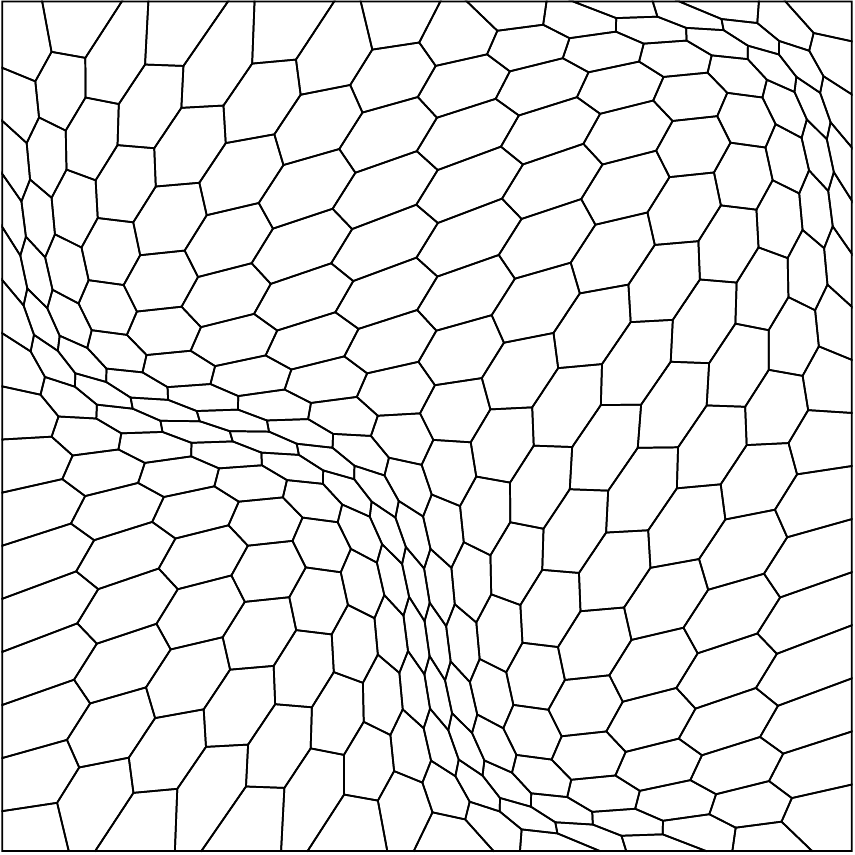}}
     \subcaptionbox{$\mathcal{T}_{h,3}$\label{fig:nonconvex_mesh}}
  {\includegraphics[width=0.32\textwidth]{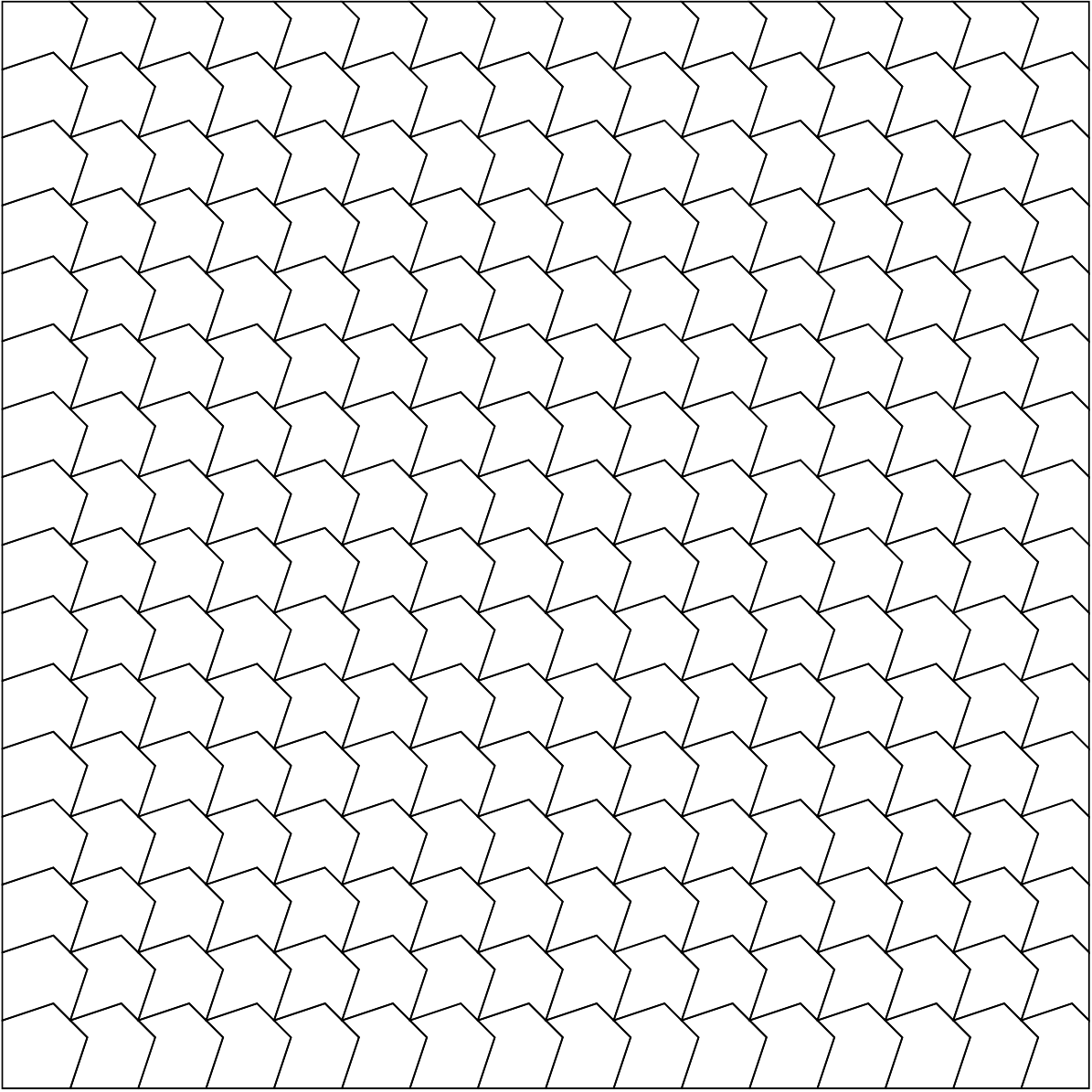}}
   \caption{Sample meshes for numerical tests: (a) structured hexagonal mesh;
    (b) transformed  hexagonal mesh; (c)  mesh with non-convex elements.}\label{fig:mesh}
\end{figure}

One of the merits of VEM is its ability  to use arbitrary polygonal meshes. To illustrate the generality of our theoretical results in term of the flexibility of VEM, we test our numerical example using three different types of polygonal meshes. 
The first level of each type of meshes  are plotted in  Figure \ref{fig:mesh}.  
The first type  of mesh $\mathcal{T}_{h,1}$ is uniform hexagonal mesh.
The second type of mesh $\mathcal{T}_{h,2}$ is generated by applying  the following coordinate transform
\begin{equation*}
\begin{split}
 x_1= \hat{x}_1+\frac{1}{10}\sin(2\pi \hat{x}_1)\sin(2\pi \hat{x}_2),\\
 x_2 = \hat{x}_2+\frac{1}{10}\sin(2\pi \hat{x}_1)\sin(2\pi \hat{x}_2),\\
\end{split}
\end{equation*}
to the uniform hexagonal  mesh $\mathcal{T}_{h,1}$.  
The third type  of mesh $\mathcal{T}_{h,3}$ is the uniform non-convex  mesh.

\subsection{Test case I: Smooth solution}  In this test, we consider the following exemplary equation with homogeneous Dirichlet boundary condition: 
\begin{equation}\label{equ:test}
-\Delta u = 2\pi^2\sin(\pi x_1)\sin(\pi x_2), \quad \text{ in  } \Omega = (0,1) \times (0,1).
\end{equation}
The exact solution is $u(\mathbf{x}) = \sin(\pi x_1)\sin(\pi x_2)$.

\begin{table}[htb!]
\centering
\caption{Numerical errors of test case I on structured hexagonal meshes }\label{tab:hex}
\resizebox{0.8\textwidth}{!}{
\begin{tabular}{|c|c|c|c|c|c|c|c|}
\hline 
Degree &N & $\|u-u^h\|_{L^{\infty}} $ & Order & $\|u-u^h\|_{W^{1,\infty}}$&Order \\ \hline
\multirow{5}{*}{k=1}  & 514 & 2.39e-03 & -- & 4.04e-01 & --  \\ \cline{2-6}
 & 2050 & 5.75e-04 & 2.06 & 2.05e-01 & 0.98  \\ \cline{2-6}
 & 8194 & 1.41e-04 & 2.03 & 1.03e-01 & 1.00  \\ \cline{2-6}
 & 32770 & 3.49e-05 & 2.02 & 5.13e-02 & 1.00  \\ \cline{2-6}
 & 131074 & 8.67e-06 & 2.01 & 2.57e-02 & 1.00  \\ \hline
\multirow{5}{*}{k=2}  & 514 & 7.72e-05 & -- & 2.38e-02 & --  \\ \cline{2-6}
 & 2050 & 9.93e-06 & 2.96 & 6.02e-03 & 1.99  \\ \cline{2-6}
 & 8194 & 1.26e-06 & 2.98 & 1.51e-03 & 2.00  \\ \cline{2-6}
 & 32770 & 1.58e-07 & 2.99 & 3.77e-04 & 2.00  \\ \cline{2-6}
 & 131074 & 1.98e-08 & 3.00 & 9.44e-05 & 2.00  \\ \hline
\multirow{5}{*}{k=3}  & 514 & 1.02e-05 & -- & 1.52e-03 & --  \\ \cline{2-6}
 & 2050 & 5.93e-07 & 4.11 & 1.94e-04 & 2.98  \\ \cline{2-6}
 & 8194 & 3.56e-08 & 4.06 & 2.44e-05 & 2.99  \\ \cline{2-6}
 & 32770 & 2.18e-09 & 4.03 & 3.05e-06 & 3.00  \\ \cline{2-6}
 & 131074 & 1.36e-10 & 4.00 & 3.82e-07 & 3.00  \\ \hline
\end{tabular}}
\end{table}

In the numerical test, we consider virtual element methods of degrees from 1 to 3. The numerical errors are documented in Table \ref{tab:hex} for the structured hexagonal meshes. What is striking in this table is the $\mathcal{O}(h^{k+1})$ optimal convergence rate for $L^{\infty}$ error and  $\mathcal{O}(h^{k})$ optimal convergence rate for $W^{1, \infty}$.  The observed convergence rates are consistent with the theoretical results predicted by Theorem \ref{thm:maxerror}.  Even though our theoretical results for $L^{\infty}$ error work for virtual element methods of degree $k\ge 2$, we can observe the optimal convergence results for the  linear virtual element method.

Let us now turn to the numerical results for the transformed hexagonal meshes, which is displayed in Table \ref{tab:brezzi}. Despite the unstructured nature of the mesh $\mathcal{T}_{h,2}$,  we can still observe the optimal maximal error, which demonstrates the Theorem \ref{thm:maxerror}.

For the non-convex meshes $\mathcal{T}_{h,3}$, we show the convergence history in Table \ref{tab:nonconvex}. Note that in the case the element is not always convex, which is not allowed in classical finite element methods.  Similar to the previous two tests, the same optimal convergence rates are observed as anticipated by the Theorem \ref{thm:maxerror}.

\begin{table}[htb!]
\centering
\caption{Numerical errors of test case I on transformed hexagonal meshes }\label{tab:brezzi}
\resizebox{0.8\textwidth}{!}{\normalfont
\begin{tabular}{|c|c|c|c|c|c|c|c|}
\hline 
Degree &N & $\|u-u^h\|_{L^{\infty}} $ & Order & $\|u-u^h\|_{W^{1,\infty}}$&Order \\ \hline
\multirow{5}{*}{k=1}  & 514 & 6.64e-03 & -- & 4.27e-01 & --  \\ \cline{2-6}
 & 2050 & 1.64e-03 & 2.02 & 2.12e-01 & 1.02  \\ \cline{2-6}
 & 8194 & 4.01e-04 & 2.03 & 1.05e-01 & 1.02  \\ \cline{2-6}
 & 32770 & 9.91e-05 & 2.02 & 5.18e-02 & 1.01  \\ \cline{2-6}
 & 131074 & 2.46e-05 & 2.01 & 2.58e-02 & 1.01  \\ \hline
\multirow{5}{*}{k=2}  & 514 & 3.12e-04 & -- & 3.98e-02 & --  \\ \cline{2-6}
 & 2050 & 4.12e-05 & 2.93 & 1.07e-02 & 1.89  \\ \cline{2-6}
 & 8194 & 5.09e-06 & 3.02 & 2.69e-03 & 2.00  \\ \cline{2-6}
 & 32770 & 6.22e-07 & 3.03 & 6.70e-04 & 2.00  \\ \cline{2-6}
 & 131074 & 7.66e-08 & 3.02 & 1.67e-04 & 2.00  \\ \hline
\multirow{5}{*}{k=3}  & 514 & 3.26e-05 & -- & 2.07e-03 & --  \\ \cline{2-6}
 & 2050 & 2.01e-06 & 4.03 & 2.89e-04 & 2.85  \\ \cline{2-6}
 & 8194 & 1.25e-07 & 4.01 & 3.59e-05 & 3.01  \\ \cline{2-6}
 & 32770 & 7.76e-09 & 4.01 & 4.48e-06 & 3.00  \\ \cline{2-6}
 & 131074 & 4.83e-10 & 4.01 & 5.57e-07 & 3.01  \\ \hline
\end{tabular}}
\end{table}

\begin{table}[htb!]
\centering
\caption{Numerical errors of test case I on  non-convex meshes}\label{tab:nonconvex}
\resizebox{0.8\textwidth}{!}{\normalfont
\begin{tabular}{|c|c|c|c|c|c|c|c|}
\hline 
Degree &N & $\|u-u^h\|_{L^{\infty}} $ & Order & $\|u-u^h\|_{W^{1,\infty}}$&Order \\ \hline
	\multirow{5}{*}{k=1} & 833 & 7.47e-03 & -- & 3.53e-01 & --  \\ \cline{2-6}
   & 3201 & 1.89e-03 & 2.04 & 1.79e-01 & 1.01  \\ \cline{2-6}
 & 12545 & 4.75e-04 & 2.02 & 8.98e-02 & 1.01  \\ \cline{2-6}
 & 49665 & 1.19e-04 & 2.01 & 4.50e-02 & 1.01  \\ \cline{2-6}
 & 197633 & 2.98e-05 & 2.01 & 2.25e-02 & 1.00  \\ \hline
\multirow{5}{*}{k=2} & 833 & 1.85e-04 & -- & 2.22e-02 & --  \\ \cline{2-6}
 & 3201 & 2.36e-05 & 3.06 & 5.63e-03 & 2.04  \\ \cline{2-6}
 & 12545 & 2.97e-06 & 3.03 & 1.41e-03 & 2.02  \\ \cline{2-6}
 & 49665 & 3.72e-07 & 3.02 & 3.54e-04 & 2.01  \\ \cline{2-6}
 & 197633 & 4.65e-08 & 3.01 & 8.85e-05 & 2.01  \\ \hline
 \multirow{5}{*}{k=3} & 833 & 1.35e-05 & -- & 1.36e-03 & --  \\ \cline{2-6}
 & 3201 & 8.04e-07 & 4.19 & 1.74e-04 & 3.05  \\ \cline{2-6}
 & 12545 & 4.92e-08 & 4.09 & 2.19e-05 & 3.04  \\ \cline{2-6}
 & 49665 & 3.08e-09 & 4.03 & 2.74e-06 & 3.02  \\ \cline{2-6}
 & 197633 & 1.94e-10 & 4.00 & 3.43e-07 & 3.01  \\ \hline
\end{tabular}}
\end{table}

\subsection{Test case II: Problem with a Gaussian surface} In this test, we consider the following Poisson equation
\begin{equation}\label{equ:test2}
-\Delta u = f, \quad \text{ in  } \Omega = (0,1) \times (0,1),
\end{equation}
with non-homogenous boundary condition $u|_{\partial\Omega}=g$. The right hand side  function $f(\mathbf{x})$  and the  boundary condition $g(\mathbf{x})$ can be calculated from the exact solution  $u(\mathbf{x}) = \exp(-\ell((x_1-0.5)^2 + (x_2-0.5)^2))$. When $\ell$ is large, the function $u$ has a rapidly varying gradient.

In this test, we select $\ell = 25$. The numerical errors on  non-convex meshes are displayed in Table \ref{tab:nonconvex_test2}. Similar to the previous test case, we can observe the optimal  convergence rates in maximal norm errors. But it requires a little bit finer meshes to observe the perfect optimal convergence rates since the rapidly varying gradient.  We also test it on other two polygonal meshes, which also gives us the same results.

\begin{table}[htb!]
\centering
\caption{Numerical errors of test case II on  non-convex meshes}\label{tab:nonconvex_test2}
\resizebox{0.8\textwidth}{!}{\normalfont
\begin{tabular}{|c|c|c|c|c|c|c|c|}
\hline 
Degree &N & $\|u-u^h\|_{L^{\infty}} $ & Order & $\|u-u^h\|_{W^{1,\infty}}$&Order \\ \hline
	\multirow{5}{*}{k=1} 
 & 833 & 2.80e-02 & -& 7.06e-01 & --\\ \cline{2-6}
 & 3201 & 7.54e-03 & 1.95 & 2.68e-01 & 1.44  \\ \cline{2-6}
 & 12545 & 1.92e-03 & 2.00 & 1.10e-01 & 1.31  \\ \cline{2-6}
 & 49665 & 4.83e-04 & 2.01 & 5.06e-02 & 1.13  \\ \cline{2-6}
 & 197633 & 1.21e-04 & 2.01 & 2.46e-02 & 1.04  \\ \hline
\multirow{5}{*}{k=2}
 & 833 & 1.08e-03 & -- & 1.94e-01 & --  \\ \cline{2-6}
 & 3201 & 1.41e-04 & 3.02 & 5.51e-02 & 1.87  \\ \cline{2-6}
 & 12545 & 1.78e-05 & 3.03 & 1.42e-02 & 1.98  \\ \cline{2-6}
 & 49665 & 2.23e-06 & 3.02 & 3.59e-03 & 2.00  \\ \cline{2-6}
 & 197633 & 2.79e-07 & 3.01 & 8.99e-04 & 2.00  \\ \hline
 \multirow{5}{*}{k=3}
 & 833 & 5.77e-04 & -- & 2.27e-02 & --  \\ \cline{2-6}
 & 3201 & 3.86e-05 & 4.02 & 3.34e-03 & 2.85  \\ \cline{2-6}
 & 12545 & 2.45e-06 & 4.03 & 4.37e-04 & 2.98  \\ \cline{2-6}
 & 49665 & 1.54e-07 & 4.02 & 5.52e-05 & 3.01  \\ \cline{2-6}
 & 197633 & 9.64e-09 & 4.01 & 6.92e-06 & 3.01  \\ \hline
\end{tabular}}
\end{table}

\section{Conclusion}
\label{sec:con}
In this paper,  we consider the error estimations in the maximum norm for virtual element methods. We establish the optimal maximum norm error estimations as to the classical numerical methods. In special, we show $\mathcal{O}(h^k)$ order convergence between the exact gradient and the gradient of the projection of virtual element solution for $k$th order virtual element methods.  When $k\ge 2$, we prove the optimal $\mathcal{O}(h^{k+1})$ order convergence for $L^{\infty}$ error.  We present a numerical example on both convex and non-convex general polygonal meshes to support our theoretical results. 


\section*{Acknowledgment}
 W. H. was partially supported by the National Natural Science Foundation of China under grants 11671304 and 11771338.
H. G. was partially supported by Andrew Sisson Fund of the University of Melbourne.\bibliographystyle{siamplain}
\bibliography{mybibfile}
\end{document}